\documentclass{article}
\usepackage{dutchcal}

\newcommand{\mr}[1]{\mathring{#1}}
\usepackage{graphics}
  \usepackage{amssymb}
    \usepackage{amsthm}
  \usepackage{amsmath}
  \usepackage{pgf}
\usepackage{tikz}
\usetikzlibrary{arrows,automata}
\usepackage{circuitikz}
\usepackage{epstopdf}
  
\def\p{\partial}

\newtheorem{theorem}{Theorem}[section]
\newtheorem{lemma}[theorem]{Lemma}
\newtheorem{corollary}[theorem]{Corollary}

\newtheorem{remark}[theorem]{Remark}

\newcommand {\bay}{\begin{array}}
\newcommand {\eay}{\end{array}}
\newcommand {\bdm}{\begin{displaymath}}
\newcommand {\edm}{\end{displaymath}}

\newcommand{\md}{\mathrm{d}}

  \makeatletter
  \let\c@equation\undefined
  \let\c@section\undefined
  \let\c@subsection\undefined
  \let\c@zad\undefined
  \makeatother
  \newcounter{section}

  \newcounter{equation}[section]
  
  \newcounter{subsection}[section]

\newcommand{\beq}{\begin{equation}}
\newcommand{\eeq}{\end{equation}}

\newcommand{\papap}{\papa \end{proof}}

\newfont{\smoldita}{cmmib8}
\newfont{\boldita}{cmmib10}
\newfont{\bboldita}{cmmib10}

\newcommand{\nn}{\nonumber}
\newcommand{\e}{\epsilon}

\newcommand{\lio}[1]{\lim_{{#1}\to 0}}
\newcommand{\lii}[1]{\lim_{{#1}\to \infty}}
\newcommand {\sem}[1]{\mbox{$({#1}(t))_{t \geq 0}$}}

\newcommand{\cl}[2]{\int\limits_{#1}^{#2}}

\newcommand{\ti}[1]{\tilde{#1}}

\newcommand{\la}{\lambda}

\newfont{\llmt}{cmmib10 scaled\magstep2}
\newfont{\lmt}{cmmib10 scaled\magstep1}
\newfont{\mt}{cmmib10}
\newfont{\smt}{cmmib8}
\newfont{\bgr}{cmmib9 scaled\magstep1}

\newcommand{\mbb}[1]{\mathbb{#1}}
\newcommand{\mc}[1]{\mathcal{#1}}
\newcommand{\mdm}[1]{\mathrm{#1}}

\begin{document}

\begin{center}{\Large Growth--fragmentation--coagulation equations with unbounded coagulation kernels}\end{center}

\begin{center}{J. Banasiak\footnote{The research has been partially supported by the National
Science Centre of Poland Grant 2017/25/B/ST1/00051 and the National Research Foundation of South Africa Grant 82770}\footnote{The authors are grateful to Prof. Mustapha Mokhtar-Kharroubi for  fruitful discussions.} \\ \small{Department of Mathematics and Applied Mathematics, University of Pretoria}\\ \small{Institute of Mathematics,  \L\'{o}d\'{z} University of Technology}\\ \small{International Scientific Laboratory of
Applied Semigroup Research, South Ural
State University}\\ \small{e-mail: jacek.banasiak@up.ac.za}\\\& \\W. Lamb\\
 \small{Department of Mathematics and  Statistics,  University of Strathclyde} \\\small{e-mail: w.lamb@strath.ac.uk}}\end{center}

\begin{center}{
}
\end{center}

\begin{abstract}
In this paper we prove the global in time solvability of the continuous growth--fragmentation--coagulation equation with unbounded coagulation kernels, in spaces of functions having finite moments of sufficiently high order. The main tool is the recently established result on moment regularization of the linear growth--fragmentation semigroup that allows us to consider coagulation kernels whose growth for large clusters is controlled by how good the regularization is, in a similar manner to the case when the semigroup is analytic.
\end{abstract}


\section{Introduction}\label{intro}

Coagulation and fragmentation play a fundamental role in a number of diverse phenomena arising both in  natural science and in industrial processes. Specific examples can be found in ecology, human biology, polymer and aerosol sciences,  astrophysics and the powder production industry; see \cite{BLL} for further details and references.  A  feature shared by these examples is that each involves  an identifiable  population of inanimate or animate objects that are capable of forming larger or smaller objects through, respectively,  coalescence or breakup. The earliest mathematical investigation into processes governed by coagulation or fragmentation was carried out by Smoluchowski in two papers \cite{Smoluch, Smoluch17},  published in 1916 and 1917.  Smoluchowski introduced, and investigated,  a coagulation model in the form of an infinite set of ordinary differential equations that describes the time-evolution of a system of particle clusters that, as a result of Brownian motion,  become sufficiently close to enable binary coagulation of clusters to occur. In this discrete-size model, it is assumed that the clusters  are comprised of a finite number of identical fundamental particles, and so a discrete (positive integer) variable can be used to distinguish between cluster sizes. Over the past one hundred years, the pioneering work of Smoluchowski has been extended considerably, and various models, both deterministic and stochastic, and  incorporating  both coagulation and fragmentation,  have been produced and studied.

In certain applications, such as droplet growth in clouds and fogs \cite{Schu40,Scot68}, where  it is more realistic to have a continuous particle size variable which  can take any positive real value, the standard deterministic coagulation-fragmentation (C-F) model is given by
\begin{equation}\label{contscfeqn}
\partial_t f(x,t)  = {\mathfrak F}f(x,t) +
\mathfrak{K}f(x,t)\ , \ \ (x,t)\in \mbb{R}_+^2\ ,  \ \  f(x,0)  =  \mr f(x)\ , \ \  x\in \mathbb{R}_+\ ,
\end{equation}
where $\mathbb{R}_+ : = (0,\infty)$, and
\begin{align}
{\mathfrak F}f(x,t) &= -a(x) f(x,t) + \  \int_x^{\infty}a(y)b(x,y) f(y,t)\,\md y \,, \label{wlcontsfrag}\\
\mathfrak{K}f(x,t) &=  \frac{1}{2}\,\int_0^x
k(x-y,y)f(x-y,t)f(y,t)\,\md y - f(x,t)\,\int_0^{\infty}
k(x,y)f(y,t)\,\md y \, \label{wlcontscoag}
\end{align}
model fragmentation and coagulation respectively; see \cite{vizi89}. Here, it is assumed that only a single size variable, such as particle mass,  is required to differentiate between  the reacting particles, with $f(x,t)$ denoting the density of particles of size $x \in \mbb{R}_+$ at time $t \geq 0$. The coagulation kernel $k(x,y)$ gives the rate at which particles of size $x$ coalesce with particles of size $y$, and  $a(x)$ represents the overall rate of fragmentation of an $x$-sized particle.  The coefficient $b(x,y)$, often called the fragmentation kernel  or daughter distribution function, can be interpreted as giving the number
of  size $x$ particles  produced by the fragmentation of a size $y$
particle; more precisely, it is the distribution function of the
sizes of the daughter particles. In most investigations into \eqref{contscfeqn},  $b$ is assumed to be  nonnegative and measurable, with $b(x,y)=0$ for $x >y$ and
\begin{equation}
\int_0^y xb(x,y)\,\md x = y, \ \mbox{ for each } y > 0,  \label{baleq1}
\end{equation}
 but is otherwise arbitrary. Note that equation~
 \eqref{baleq1} can be viewed as a local mass conservation property, as it expresses the fact that, when the size variable is taken to be particle mass,  the total mass of all the daughter particles produced by a fragmentation event is the same as that of the parent particle.

 In the case of deterministic models, either discrete or continuous size, two main approaches have been used extensively in their analysis, with one involving weak compactness arguments and the other utilising the well-developed theory of operator semigroups. Comprehensive treatments of each are given in \cite{BLL}, and there is also an excellent account in \cite[Chapter 36]{bobrowski2016convergence} of the semigroup approach to the discrete C-F equation.  We focus here on the application of semigroup techniques to continuous C-F models, where the strategy is to express the pointwise initial-value problem \eqref{contscfeqn} as a semilinear abstract Cauchy problem (ACP) of the form
 \begin{equation}\label{ACPFK}
  \frac{d}{dt}f(t) = Ff(t) + Kf(t),\ t \in \mathbb{R}_+; \quad f(0) = \mr f,
  \end{equation}
 posed in a physically relevant  Banach space $X$.  In \eqref{ACPFK}, $F$ and $K$ are operator realisations in $X$ of the formal expressions
 \begin{align}
(\mathcal{F}f)(x) &:=  -a(x) f(x) + \  \int_x^{\infty}a(y)b(x,y)f(y)\,\md y,\ x \in \mbb{R}_+, \label{formalF}\\
(\mathcal{K}f)(x) &:= \frac{1}{2}\,\int_0^x\!\!
k(x-y,y)f(x-y)f(y)\,\md y \!-\! f(x)\,\int_0^{\infty}\!\!
k(x,y)f(y)\,\md y,\ x \in \mbb{R}_+. \label{wlcontscoag}
\end{align}
Initially, only the linear fragmentation part of \eqref{ACPFK} is examined, and a representation $F$ is sought such that $F$ generates a  strongly continuous semigroup \sem{S_{F}} on $X$. If this is possible, then the full abstract C-F problem is recast as the fixed point equation
\begin{equation}\label{Fixedpt}
f(t) = S_{F}(t)\mr f  + \int_0^t S_{F}(t-s)Kf(s)\,\md s, \ t \in \mathbb{R}_+,
\end{equation}
to which standard results can be applied to yield the existence and uniqueness of  mild and classical solutions
$f: [0,\tau_{\max}) \to X$.  The identification $[f(t)](x) = f(x,t)$ then leads, after some further analysis, to a solution of the pointwise problem
 \eqref{contscfeqn}.

 Historically, the  semigroup approach to C-F problems originated in 1979 with the publication of  a  seminal paper by Aizenman and Bak \cite{AizBak} for the specific case where the coagulation kernel $k$  is constant, and the fragmentation rate and the fragmentation kernel  are given by $a(x) = x$ and $b(x,y) = 2/y$.  The work presented in \cite{AizBak}  was  later extended in 1997 to bounded coagulation kernels and more general fragmentation rates and kernels \cite{MLM97a, LML}.  Common to these early semigroup investigations is the use of more tractable, truncated versions of the fragmentation problem to generate a sequence of semigroups that converge, in an appropriate manner, to the semigroup for the original problem; for example, see \cite[Sections~3 \& 4]{LML}.  In contrast, the year 2000 saw the introduction, in \cite{BaT01},  of  a novel approach to the fragmentation problem  that relies on  the theory of substochastic semigroups. In recent years, this substochastic semigroup approach has been developed further and used to prove  many important properties of the fragmentation semigroup such as its analyticity and, in the discrete case, compactness, \cite{BaLa12a, BaLa12b}. These properties have made it possible to extend earlier semigroup derived results on the well-posedness of C-F equations to the case where the coagulation kernel may be unbounded; see \cite{BaLa12a, BLL13, Ban2020}. Moreover, it is shown in \cite{Ban2020} that whenever the semigroup and weak compactness approaches are both applicable to a C-F problem, they both lead to the same solutions.

With regard to the choice of an appropriate space $X$, the early semigroup (and also weak compactness) analyses of \eqref{contscfeqn} used the spaces  $X_0 := L_1(\mathbb{R}_+, \md x)$, $X_1 := L_1(\mathbb{R}_+, x\md x)$ and also $X_{0,1} := L_1(\mathbb{R}_+, (1+x) \md x)$, with respective norms
\[
\|f\|_{[0]} := \int_0^\infty |f(x)| \md x; \ \|f\|_{[1]} := \int_0^\infty |f(x)|x \md x; \ \|f\|_{[0,1]} := \int_0^\infty |f(x)| (1+x) \md x.
\]
These spaces  were chosen due to the fact that, for a nonnegative solution $f$ of \eqref{ACPFK}, $\|f(t)\|_{[0]}$ gives the total number of particles in the system, while $\|f(t)\|_{[1]}$ gives the total mass. However, in later investigations it was found that improved results could be obtained by imposing some additional control on the evolution of large particles. A convenient way of introducing such a control is to consider the C-F problem in the more general weighted $L_1$ spaces  $X_{m}:= L_1(\mathbb{R}_+, x^m \md x)$ and $X_{0,m}:=L_1(\mathbb{R}_+, (1+ x^m) \md x)$.  The norms on these spaces are defined by
\begin{equation}\label{norms}
\|f\|_{[m]} := \int_0^\infty |f(x)|x^m \md x; \ \|f\|_{[0,m]} := \int_0^\infty |f(x)|w_m(x) \md x, \
\end{equation}
 where  $w_m(x) := 1+x^m$. We shall also use the notation
\begin{equation}\label{Moments}
M_m(t):= \int_0^\infty f(x,t)x^m\,\md x; \  \ M_{0,m}(t) := \int_0^\infty f(x,t) w_m(x)\md x,
\end{equation}
when discussing the norms of  nonnegative solutions to \eqref{contscfeqn}.  Clearly, $M_m(t)$ and $M_{0.m}(t)$ are finite provided $f(\cdot, t) \in X_m$ and $f(\cdot,t) \in X_{0,m}$.

For ease of exposition, we have restricted our attention in the above discussion to situations involving only the opposing  processes of  fragmentation and coagulation, and in which the total mass in the system of particles should be a conserved quantity. In many cases, however, these two processes may be complemented by other events which can change the total mass in the system.  For example, mass loss can arise due to oxidation, melting, sublimation and dissolution of matter on the exposed particle surfaces. The reverse process of mass gain can also occur due to the precipitation of matter from the environment.  Continuous coagulation and fragmentation processes, combined with a mass transport term that leads to either mass loss or mass gain, have also been studied using functional analytic and, in particular, semigroup methods; for example, see \cite{BaAr, BaLa09, Bana12a} and \cite[Section 5.2]{BLL}, or \cite{DoGa10, Ber2019, PerTr} where, however, the focus is on the long-term behaviour of the linear growth-fragmentation processes. The discrete version of the models have been comprehensively analysed in \cite{Banasiak2018, Banasiak2019}. In the case when the growth rate of a particle of mass $x$ is $r(x)$, the appropriate modified version of \eqref{contscfeqn} is
\begin{align}
\partial_t f(x,t)  &= - \p_x[r(x)f(x,t)] + {\mathfrak F}f(x,t) +
\mathfrak{K}f(x,t)\ , \ \ (x,t)\in \mbb{R}_+^2\ ,\nn\\  f(0,x) & =  \mr f(x)\ , \ \  x\in \mathbb{R}_+\ .\label{initprof}
\end{align}
The main goal of the paper is to prove global classical solvability of \eqref{initprof} in the spaces $X_{0,m}$ for sufficiently large $m,$ when the coagulation rate $k$ is unbounded (though controlled by the fragmentation rate). In this way we extend  the results of \cite{Bana12a}, where only bounded coagulation operators were considered. We use the standard semigroup theory based approach of re-writing  \eqref{initprof} as an abstract Volterra equation with the kernel given by the linear growth--fragmentation semigroup. The main tool is the moment improving property of this semigroup, proven in \cite{Ber2019}, that makes it a little like an analytic semigroup and allows for an approach similar to that used in \cite{BaLa12a, Ban2020} for pure fragmentation--coagulation problems, where the fragmentation semigroup is indeed analytic. In other words, the growth--fragmentation semigroup retains the moment regularization property of the fragmentation semigroup but,  since it is not regularizing with respect to the differentiation operator, it fails to be  analytic. Thus, while the well-posedness proof for \eqref{initprof} follows standard steps, particular estimates must be tailor made for this specific case to yield the desired result. More precisely, while the existence of the mild solution is obtained by a typical fixed point argument, the involved integral operator is weakly singular, in contrast to the standard theory where it is assumed to be continuous, see e.g. \cite[Theorem 6.1.2]{Pa}.  Similarly, the proof that the mild solution is a classical solution cannot be obtained, as in other cases where unbounded nonlinearities occur, by using the differentiability of the semigroup, since the growth--fragmentation semigroup is not analytic. Instead, the approach we adopt is to follow  \cite[Theorem 6.1.5]{Pa}, where  a regularity result is established for the case of a continuous nonlinearity,  but  again we have to show that the result can be extended to an appropriately restricted singular nonlinearity.

The paper is organized as follows. Section 2 deals with the linear growth--fragmentation equation. In particular, we use the Miyadera perturbation theorem to show that the growth--fragmentation operator is the generator of a positive semigroup on $X_{0,m}$ and provide a precise characterization of its domain, without imposing any restriction on the behaviour of the growth rate $r$ at $x=0.$ In this way we improve the corresponding results of  \cite{Bana12a, Ber2019}. The improved generation theorem is further used to slightly simplify the proof of the moment regularization property, given in \cite{Ber2019}. Section 3 is devoted to the full equation \eqref{initprof}. The existence of local mild and classical solutions is proved under quite general conditions, while the global solvability, done along the lines of \cite{Ban2020}, requires some additional assumptions to control the growth term.


\section{Fragmentation with growth}

Adopting the semigroup based strategy described in Section 1, we begin our analysis of equation \eqref{initprof} by considering the linear equation that is obtained on ignoring the coagulation terms. For technical reasons, which will become clear later, it is convenient to introduce an additional absorption term, $-a_1f$. This results in the linear equation
\begin{align}
\begin{split}
 \p_t f(x,t) &=
 -\p_x[r(x)f(x,t)] -q(x)f(x,t) + \cl{x}{\infty}a(y)b(x,y)f(y,t)\,\mathrm{d}y,\ \ (x,t) \in \mbb{R}_+^2,\\
 f(x,0)&=\mr f(x),\ \ x \in \mathbb{R}_+,
 \label{reml}
 \end{split}
  \end{align}
  where $q(x) = a(x)+a_1(x)$.
The aim is to express \eqref{reml} as an ACP of the form
\begin{equation}\label{ACPFG}
\frac{d}{dt}f(t)= T_{0,m}f(t) + B_{0,m}f(t), \ t > 0;\ \  f(0) = \mr f,
\end{equation}
where $T_{0,m}$ and $B_{0,m}$, respectively, are  operator realisations in $X_{0,m}$ of the formal expressions
\begin{equation}\label{formalTB}
(\mathcal{T}f)(x) := -\p_x[r(x)f(x)] -q(x)f(x); \ \  (\mathcal{B}f)(x) := \cl{x}{\infty}a(y)b(x,y)f(y)\,\mathrm{d}y.
\end{equation}
The ACP \eqref{ACPFG} will be   well posed in $X_{0,m}$ provided the operator $G_{0,m}:= T_{0,m}+B_{0,m}$  is the infinitesimal generator of a strongly continuous semigroup, \sem{S_{G_{0,m}}}, on $X_{0,m}$. To show that it is possible, under suitable conditions, to define such an operator $G_{0,m}$,  we first use the  Hille-Yosida theorem to establish that $T_{0,m}$, when defined appropriately,  generates a strongly continuous semigroup,  \sem{S_{T_{0,m}}} (the absorption semigroup), on $X_{0,m}$. The operator $B_{0,m}$  is then shown to be a Miyadera perturbation of $T_{0,m}$ and this leads immediately to the existence of \sem{S_{G_{0,m}}}.

\subsection{The absorption semigroup}

The transport part of the problem is given by
\begin{align}
\begin{split}
 \p_t f(x,t) &=
 -\p_x[r(x)f(x,t)] -q(x)f(x,t), \ \ (x,t) \in \mbb{R}_+,\\
 f(x,0)&=\mr f(x),\ \  x \in \mbb{R}_+,
 \label{remla}
 \end{split}
  \end{align}
where, as stated above,  $q = a + a_1$.  We assume throughout that the fragmentation and growth rates,  $a$  and $r$ respectively, satisfy
\begin{eqnarray}
&& 0 \leq a \in L_{\infty,loc}([0,\infty)); \label{aloc}\\
&& 1/r \in L_{1,loc}(\mathbb{R}_+)  \mbox{ and } 0<r(x)\leq r_0+r_1 x \leq \ti r(1+x)\quad {\rm on}\quad \mbb R_+, \label{fmlras}
\end{eqnarray}
for some nonnegative constants $r_0,r_1$ and $\ti r =\max\{r_0,r_1\}$. With regard to the additional absorption term, $a_1$,  it is assumed that
\begin{equation}\label{a1con}
0 \leq a_1 \in L_{\infty,loc}([0,\infty)) \mbox{ and } a_1(x)/a(x) \mbox{ remains bounded as } x \to \infty.
\end{equation}
On defining operators $A_{0,m}$ and $A^{(1)}_{0,m}$ on their maximal domains in $X_{0,m}$ by
\begin{eqnarray}
&& A_{0,m}f : = -af; \  \ D(A_{0,m}):= \{f \in X_{0,m} : af \in  X_{0,m}\}, \label{defnA}\\
&& A^{(1)}_{0,m}f : = -a_1f; \ \ D(A^{(1)}_{0,m}):= \{f \in X_{0,m} : a_1f \in  X_{0,m}\}, \label{defnA1}
\end{eqnarray}
the second assumption in \eqref{a1con} guarantees that $D(A_{0,m}) \subseteq D(A^{(1)}_{0,m})$.

In the following treatment of \eqref{remla} we have to distinguish between two distinct cases that may arise due the behaviour of $r(x)$ close to $x = 0$. If we use the symbol $\cl{0^+}{}$ to denote an integral in some right neighbourhood of $0$, then we may have either
\begin{equation}
\cl{0^+}{}\frac{\md x}{r(x)}=+\infty, \label{assr2}
\end{equation}
or
\begin{equation}
\cl{0^+}{}\frac{\md x}{r(x)}<+\infty. \label{assr1}
\end{equation}
When \eqref{assr2} is satisfied,  the characteristics associated with the transport equation do not reach $x=0$ and therefore the problem does not require a boundary condition to be specified. This case has been thoroughly researched in \cite{Bana12a, BLL}, and, as in \emph{op. cit.}, we define $T_{0,m}$ by
\begin{equation}\label{dtkmax}
T_{0,m}f : = \mathcal{T}f; \ \ D(T_{0,m})  := \left\{f \in X_{0,m}\; :\; rf \in AC(\mbb R_+)\;{\rm
and}\;\frac{d}{dx}(rf), qf \in X_{0,m}\right\},
\end{equation}
where $AC(\mbb R_+)$ denotes the class of functions that are absolutely continuous on all compact subintervals of $\mbb{R}_+$.  On the other hand, when \eqref{assr1} holds, the characteristics do reach $x = 0$ and therefore a boundary condition is required. Here, following \cite{Ber2019}, we impose the homogeneous condition
\begin{equation}
\lim\limits_{x\to 0^+} r(x)f(x,t) = 0
\label{bc1}
\end{equation}
 but note that more general cases can also be considered, \cite{BaLa09}. It follows that  $D(T_{0,m})$ is then given by
 \begin{align}
 D(T_{0,m})&:= \left\{f \in X_{0,m}\; :\; rf \in AC(\mbb R_+),\, \frac{d}{dx}(rf), qf \in X_{0,m}\; \right.\nn\\&\left.\mbox{ and } r(x)f(x) \to 0 \mbox{ as } x \to 0^+ \right\}.\label{dtbcmax}
 \end{align}
  To make the Hille-Yosida theorem applicable, we must determine the resolvent operator, $R(\lambda, T_{0,m})$. Following \cite[Section 5.2]{BLL}, we begin by solving
 \begin{equation}\label{resolveq}
 \lambda f(x) + \frac{d}{dx}(r(x)f(x)) + q(x)f(x) = g(x),\ \ x \in \mathbb{R}_+,
 \end{equation}
 where $g \in X_{0,m}$.  On introducing antiderivatives, $R$ and $Q$, of $1/r$ and $q/r$ respectively, defined on $\mathbb{R}_+$ by
 \begin{equation}
R(x) := \cl{1}{x}\frac{1}{r(s)}\md s, \qquad Q(x):=
\cl{1}{x}\frac{q(s)}{r(s)}\md s,
\label{RQ}
\end{equation}
we can proceed formally to obtain the general solution of \eqref{resolveq} in the form
\begin{equation}
f(x) = v_\la(x)\, \cl{0}{x}e^{\la
R(y)+Q(y)}g(y) \mdm{d}y + C\,v_\la(x),
\label{Ph10}
\end{equation}
where $C$ is an arbitrary constant and
\begin{equation}\label{vlambda}
v_\la(x) = \frac{e^{-\la R(x)-Q(x)}}{r(x)},\ \ x \in \mathbb{R}_+.
\end{equation}
An immediate consequence of \eqref{fmlras} and (\ref{RQ}) is that
$R$ is strictly increasing (and hence invertible) on $\mathbb{R}_+$,
and $Q$ is nondecreasing on $\mathbb{R}_+$. Consequently, if we define
\begin{equation}\begin{split}
\lio{x} R(x) =: m_R,&\quad \lii{x} R(x) =: M_R,\\
\lio{x} Q(x) =: m_Q, &\quad\lii{x} Q(x) =: M_Q,\label{mr}\end{split}
\end{equation}
then $m_R$ is finite and negative if \eqref{assr1} holds and $m_R=-\infty$ otherwise. Furthermore,  $M_R=\infty$ due to \eqref{fmlras}, whereas  $m_Q$ and $M_Q$ can be finite or infinite depending on the interplay between  $q$ and $r$.  In what follows we need the following result that is a slight modification of \cite[Lemma 2.1]{BaLa09} and \cite[Corollary 5.2.9]{BLL}.

\begin{lemma}\label{lemomrm}
Let $m\geq 1$ be fixed and define  $\omega_{r,m}:=2m\ti r$, where $\ti r$ is the positive constant in \eqref{fmlras}. Then, for any $\la> \omega_{r,m}$  and  $0<\alpha<\beta \leq \infty$,
\begin{align}
&I_{0,m}(\alpha,\beta):=\cl{\alpha}{\beta} \frac{e^{-\la R(s)}}{r(s)}
w_m(s)\md s \,\leq \, \frac{e^{-\la
R(\alpha)}}{\la-\omega_{r,m}}\,w_m(\alpha)\,, \label{ixest0m}\\
&J_{0,m}(\alpha,\beta) := \cl{\alpha}{\beta}\frac{(\la+q(s))e^{-\la R(s)-Q(s)}}{r(s)}w_m(s)\md s \,\leq \, \frac{\la e^{-\la
R(\alpha)-Q(\alpha)}}{\la-\omega_{r,m}}\,w_m(\alpha),\label{Jest0m}
\end{align}
where, as in \eqref{norms}, $w_m(x) = 1+ x^m.$
\end{lemma}

\proof  Since
\[
I_{0,m}(\alpha,\beta) = -\frac{1}{\la}\cl{\alpha}{\beta}w_m(s)\frac{d}{ds} e^{-\la R(s)}\md s,
\]
we can integrate by parts, and then use \eqref{fmlras}, to obtain

\begin{eqnarray}
&&I_{0,m}(\alpha,\beta)=  \frac{1}{\la}e^{-\la R(\alpha)}w_m(\alpha) - \frac{1}{\la}e^{-\la R(\beta)}w_m(\beta) +\frac{m}{\la}\cl{\alpha}{\beta}e^{-\la R(s)}s^{m-1}\md s\nn\\
&&\phantom{xx}\leq \frac{1}{\la}e^{-\la R(\alpha)}w_m(\alpha) +\frac{m\ti r}{\la}\cl{\alpha}{\beta}\frac{e^{-\la R(s)}}{r(s)}(1+s)s^{m-1}\md s.\nn
\end{eqnarray}
The inequality $(1+s)s^{m-1} \leq 2(1+s^m)$, which holds for all $s > 0$ and each fixed $m \geq 1$, yields
\begin{equation}
 I_{0,m}(\alpha,\beta) \leq \frac{1}{\la}e^{-\la
R(\alpha)}w_m(\alpha)+\frac{2m\ti r}{\la}I_{0,m}(\alpha,\beta),\label{Ik}
\end{equation}
and  \eqref{ixest0m} follows.

To prove \eqref{Jest0m},   we note first that
\begin{equation}
 \frac{\la + q(x)}{r(x)}e^{-\la R(x) -Q(x)} =  - \frac{d}{dx}e^{-\la R(x) -Q(x)}. \label{difeeab}
\end{equation}
Hence, on integrating by parts and using \eqref{fmlras}, together with the monotonicity of $e^{-Q}$, we obtain similarly to \eqref{Ik},
\begin{eqnarray}
 J_{0,m}(\alpha,\beta)
&\leq& e^{-\la R(\alpha)-Q(\alpha)}w_m(\alpha)+ m\cl{\alpha}{\beta}e^{-\la R(s)-Q(s)}s^{m-1}\md s\nn \\
&\leq& e^{-\la R(\alpha)-Q(\alpha)}w_m(\alpha)+ 2m\ti r e^{-Q(\alpha)}\cl{\alpha}{\beta}\frac{e^{-\la R(s)}}{r(s)}w_m(s)\md s\nn \\
&=& e^{-\la
R(\alpha)-Q(\alpha)}w_m(\alpha) +e^{-Q(\alpha)}\omega_{r,m}I_{0,m}(\alpha,\beta).
\end{eqnarray}
The stated inequality, \eqref{Jest0m}, now follows from \eqref{ixest0m}.
 \qed

\begin{lemma} \label{tgin} Let $\la>0$ and let $v_\lambda$ be defined by \eqref{vlambda}. \newline
(a) \   If \eqref{assr1} holds, then $v_\la$ does not satisfy \eqref{bc1}. \newline
(b) \  If \eqref{assr2} holds, then $v_\la\notin X_{0,m}$ for any $m \geq 1$.
\label{tgin}
\end{lemma}
\proof
(a) For $0<x<1$ we have
\[
r(x)v_\lambda(x) = e^{\cl{x}{1}\frac{\la +q(s)}{r(s)}\md s},
\]
and so $r(x)v_\lambda(x)$ does not converge to $0$ as $x \to 0^+.$

(b) Let \eqref{assr2} be  satisfied. Then, for each $\la > 0$,
\begin{equation}\label{infinitelimit}
\lim_{x \to 0^+} e^{-\la R(x)} = \lim_{x \to 0^+} e^{\cl{x}{1}\frac{\la}{r(s)}\md s} = \infty.
\end{equation}
Consequently, since $e^{-Q(x)}\geq 1$ for $x \in [0,1]$, and $R(1) =0$, we obtain
\[
\cl{0}{\infty} v_\la(x)w_m(x)\md x \geq \cl{0}{1} \frac{e^{-\la R(x)}}{r(x)} \md x = -\la \cl{0}{1} \frac{d}{dx} e^{-\la R(x)} \md x = \la (\lim\limits_{x\to 0^+} e^{-\la R(x)} - 1),
\]
and, from \eqref{infinitelimit}, it follows that $v_\lambda \notin X_{0,m}$.
 \qed

Motivated by \eqref{Ph10} and Lemma \ref{tgin},  we are led, as in \cite[Section 5.2.2]{BLL}, to
\begin{equation}
 [\mc R(\la)g](x) := \frac{e^{-\la R(x)-Q(x)}}{r(x)} \cl{0}{x}e^{\la
R(y)+Q(y)}g(y) \mdm{d}y \label{defres0}
\end{equation}
as a natural candidate for the resolvent, $R(\la,T_{0,m})$,  of $T_{0,m}$.

\begin{theorem}\label{th.5.2.11}
Let \eqref{aloc}, \eqref{fmlras} and \eqref{a1con} be satisfied.  Then, for each $m\geq 1$ and $\la > \omega_{r,m}$, the resolvent of $(T_{0,m},D(T_{0,m}))$ (in both cases \eqref{dtkmax} and \eqref{dtbcmax}) is given by $R(\la,T_{0,m}))g = \mc R(\la)g, \ g \in X_{0,m}$.  Moreover,
\begin{equation}
\| R(\la,T_{0,m})g\|_{[0,m]} \leq \frac{1}{\la-\omega_{r,m}}\,\| g \|_{[0,m]}, \ \mbox{ for all } g \in X_{0,m},
\label{resest0m}
\end{equation}
and therefore $(T_{0,m},D(T_{0,m}))$
is the generator of a strongly continuous, positive,
quasi-contractive semigroup, \sem{S_{T_{0,m}}}, on $X_{0,m}$ with type not exceeding
$\omega_{r,m}$; that is,
$$
\|S_{T_{0,m}}(t)f \|_{[0,m]}\leq e^{\omega_{r,m}t}\|f\|_{[0,m]}\,, \ \mbox{ for all } f\in X_{0,m}.
$$
\end{theorem}

\begin{proof} Let $g \in X_{0,m}$, where $m \geq 1$.  Then, by \eqref{ixest0m} and the monotonicity of $e^{-Q}$,
\begin{equation}\begin{split}
\|  \mc R(\la)g\|_{[0,m]} &\leq
\cl{0}{\infty}\frac{e^{-\la
R(x)-Q(x)}}{r(x)} \left(\cl{0}{x}e^{\la
R(y)+Q(y)}|g(y)| \mdm{d}y\right)w_m(x) \mdm{d}x \\
&=
\cl{0}{\infty}\left(|g(y)|e^{\la
R(y)+Q(y)}\cl{y}{\infty}\frac{e^{-\la R(x)-Q(x)}}{r(x)}w_m(x) \md x\right) \mdm{d}y\\
& \leq  \int_0^\infty e^{\lambda R(y)} I_{0,m}(y,\infty) |g(y)| \md y \leq \frac{1}{\la-\omega_{r,m}} \|g\|_{[0,m]}\,.\end{split}
\label{resest00}
\end{equation}
 Similarly, using  (\ref{Jest0m}) and
(\ref{resest00}), we obtain
\begin{equation}\label{2.29}
\begin{split}
\|q \mc R(\la)g\|_{[0,m]} &\leq
\cl{0}{\infty}\left({e^{\la
R(y)+Q(y)}}\cl{y}{\infty}\frac{w_m(x)q(x)e^{-\la R(x)-Q(x)}}{r(x)} \mdm{d}x\right)|g(y)| \mdm{d}y\\
&\leq
\cl{0}{\infty}{e^{\la
R(y)+Q(y)}}J_{0,m}(y,\infty)|g(y)| \mdm{d}y \leq
\frac{\la}{\la-\omega_{r,m}}\|g\|_{[0,m]}.\end{split}
\end{equation}
To establish that  $r \mc R(\la)g \in
AC(\mbb R_+)$, we observe that $\exp(-\la R-Q)$ is a bounded function
that is differentiable a.e. on $(0,\infty)$, and also that the function defined by the integral in \eqref{defres0} is absolutely continuous on $(0,\infty)$. Furthermore,  direct substitution shows that
$$
\la [\mc R(\la)g](x) + \frac{d}{dx}(r(x)[\mc R(\la)g](x)) + q(x)[\mc R(\la)g](x)  = g(x)$$
for almost all $x>0$ and hence, by \eqref{resest00} and \eqref{2.29}, $ \frac{d}{dx}(r\mc R(\la)g) \in X_{0,m}$. Since
\[
r(x)[\mc R(\la)g](x)  \to 0 \mbox{ as } x \to 0^+
\]
 whenever  \eqref{assr1} is satisfied, it follows that
$\mc R(\la)g \in D(T_{0,m})$.  On the other hand, thanks
to Lemma~\ref{tgin}, the operator $\la I-T_{0,m}$ is injective,
which shows that (\ref{defres0}) defines the resolvent of $T_{0,m}$, and the generation of a strongly continuous, positive, quasi-contractive semigroup generated by $T_{0,m}$  can then be deduced from the Hille-Yosida theorem together with the positivity of $R(\la,T_{0,m})$.
\end{proof}

 \subsection{The growth-fragmentation semigroup}

 We now consider the growth--fragmentation  equation \eqref{reml}. In addition to the restrictions \eqref{aloc}, \eqref{fmlras} and \eqref{a1con} imposed on $a,r$ and $a_1$ respectively, we assume that the fragmentation kernel, $b$, satisfies \eqref{baleq1} and further for each $m \geq 0$,   we define
\begin{align}
n_m(y)&=\cl{0}{y}b(x,y)x^m\md x, \label{nmy}\\
N_m(y)& = y^m-n_m(y).
\label{Nmy}
\end{align}
The local mass conservation condition in \eqref{baleq1} then leads to
\begin{equation}
n_0(y) > 1, \; N_m(y) > 0, \; m>1; \quad N_1(y) = 0; \quad N_m(y) <0, \; 0\leq m<1;
\label{Nm}
\end{equation}
see\cite[Eqns. (2.2.53) \& (2.3.16)]{BLL}.  The function $n_0$ is also assumed to satisfy
\begin{equation}
 n_0(y)  \le b_0 (1+y^l)\ , \qquad y \in \mathbb{R}_+\ , \label{PhPr005}
\end{equation}
for constants  $b_0 > 0$ and $l\ge 0$.
A crucial role in the analysis is played by the further assumption that there exists $m_0>1$ such that
\begin{equation}
\liminf\limits_{y\to \infty}\frac{N_{m_0}(y)}{y^{m_0}} >0.
\label{goodchar1}
\end{equation}
It follows, \cite[Theorem 2.2]{Ban2020}, that for any fixed $y>0$, $(1,\infty)\ni m \mapsto  \frac{N_{m}(y)}{y^{m}}$ is an increasing and concave function. Hence, if  \eqref{goodchar1} holds for some $m_0>1$, then
\begin{equation}
\liminf\limits_{y\to \infty}\frac{N_{m}(y)}{y^{m}} >0,
\label{goodchar}
\end{equation}
for all $m >1$. For a given $m>0$, \eqref{goodchar} yields the existence of $y_m>0$ and $c_m<1$ such that
\begin{equation}
n_m(y) \leq c_my^m, \quad y\geq y_m.
\label{bmom}
\end{equation}
The monotonicity and concavity of $(1,\infty)\ni m \mapsto  \frac{N_{m}(y)}{y^{m}}$ implies further that there is $y_0>0$ such that  for any $m>1$ there is $c_m' <1$   such that
\begin{equation}
n_m(y) \leq c'_my^m, \quad y\geq y_0
\label{bmom1}
\end{equation}
and, for any $m>1$,  we can take $c_l' = c_m'$ for $l\geq m$.
We note that \eqref{goodchar} is satisfied for a large class of fragmentation kernels $b,$ including the homogeneous ones used in \cite{Ber2019}; there are, however, cases when it does not hold, \cite[Example 5.1.51]{BLL}.

Henceforth, we  assume that
\begin{equation}
m > \max\{1,l\},\label{massump}
\end{equation}
and for each $m$ we define an operator realisation, $B_{0,m}$, of the formal expression $\mathcal{B}$ in \eqref{formalTB} by
\begin{align}
(B_{0,m}f)(x) &:= \cl{x}{\infty}a(y)b(x,y)f(y,t)dy,\ x \in \mathbb{R}_+; \nn\\
  D(B_{0,m})&:= \{f \in X_{0,m}: B_{0,m}f \in  X_{0,m}\}. \label{B0m}
\end{align}
\begin{theorem} Let \eqref{PhPr005}, \eqref{goodchar} and \eqref{massump}, together with the assumptions of Theorem  \ref{th.5.2.11}, be satisfied. Then
$(G_{0,m},D(T_{0,m})) = (T_{0,m}+B_{0,m}, D(T_{0,m}))$ generates a strongly continuous, positive semigroup, \sem{S_{G_{0,m}}}, on $X_{0,m}$.  \label{Miy}
\end{theorem}
\begin{proof}
We use a version, \cite[Lemma 5.12]{BaAr}, of a theorem due to Desch that is applicable to positive operators in $L_1$ spaces.  Thus, we must prove that  $\|B_{0,m} R(\la, T_{0,m}) \| < 1$ for some
$\la > \omega_{r,m}$.  Since $B_{0,m}R(\la, T_{0,m})$ is positive, we need only establish that
$\|B_{0,m} R(\la, T_{0,m}) f\|_{[0,m]} < \|f\|_{[0,m]}$ for all $f$ in the positive cone, $X_{0,m,+}$, and some
$\la > \omega_{r,m}$; see \cite[Proposition 2.67]{BaAr}.  Given that  $f \in X_{0,m,+}$ and $\la > \omega_{r,m}$, we have
\begin{align*}
\|B_{0,m} R(\la, T_{0,m}) f\|_{[0,m]} &= \cl{0}{\infty} \left(\cl{x}{\infty} a(y) b(x,y) [R(\la, T_{0,m}) f](y)\md y \right)w_m(x)\md x \\
&=\cl{0}{\infty}a(y) [R(\la, T_{0,m}) f](y)(n_0(y) + n_m(y))\md y,
\end{align*}
where we have used the notation introduced in \eqref{nmy}.  On setting
$a_\rho = \mathrm{ess}\!\!\!\!\sup\limits_{x \in [0,\rho]}a(x)$ for each fixed  $\rho > 0$, we obtain, by \eqref{PhPr005}, \eqref{Nm} and \eqref{resest0m},
\begin{align*}
\cl{0}{\rho}a(y) [R(\la, T_{0,m}) f](y)(n_0(y) + n_m(y))\md y&\leq a_\rho \cl{0}{\infty}[R(\la, T_{0,m}) f](y)(b_0 w_l(y) + y^m)\md y \\&\leq C_m a_\rho \cl{0}{\infty}[R(\la, T_{0,m}) f](y)w_m(y)\md y\\
&\leq \frac{C_m a_\rho}{\la-\omega_{r,m}}\|f\|_{[0,m]},
\end{align*}
where
$$C_m := \sup\limits_{0\leq y < \infty} b_0\frac{w_l(y)}{w_m(y)} + \frac{y^m}{w_m(y)} \leq 2b_0+1.$$
To obtain a suitable estimate on the integral over the infinite interval $[\rho, \infty)$, we now use  \eqref{bmom}. Since $\rho > y_m$ can be chosen sufficiently large so that
\begin{equation}
b_0\frac{w_l(y)}{w_m(y)} <\delta, \ \ \mbox{ for all } y \geq \rho,
\label{bmom2}
\end{equation}
where $c_m + \delta < 1$, we can argue as in \eqref{2.29} to obtain
\begin{align*}
&\cl{\rho}{\infty}a(y) [R(\la, T_{0,m}) f](y)(n_0(y) + n_m(y))\md y \\
&\leq (\delta + c_m)\cl{0}{\infty}a(y) [R(\la, T_{0,m}) f](y)w_m(y)\md y\\
&= (\delta + c_m) \cl{0}{\infty}\left({e^{\la
R(x)+Q(x)}}\cl{x}{\infty}\frac{w_m(y) a(y)e^{-\la R(y)-Q(y)}}{r(y)} \mdm{d}y\right)f(x) \mdm{d}x \\
&\leq (\delta + c_m) \cl{0}{\infty}\left({e^{\la
R(x)+Q(x)}}\cl{x}{\infty}\frac{w_m(y) (\la + q(y))e^{-\la R(y)-Q(y)}}{r(y)} \mdm{d}y\right)f(x) \mdm{d}x \\
&=
(\delta + c_m)\cl{0}{\infty}{e^{\la
R(x)+Q(x)}}J_{0,m}(x,\infty)f(x) \mdm{d}x \leq
\frac{\la(\delta + c_m)}{\la-\omega_{r,m}}\|f\|_{[0,m]}.
\end{align*}
Hence
\begin{align*}
\|B_{0,m} R(\la, T_{0,m}) f\|_{[0,m]} \leq \left(\frac{C_m a_\rho}{\la-\omega_{r,m}} + \frac{\la}{\la-\omega_{r,m}}(\delta + c_m)\right)\|f\|_{[0,m]}.
\end{align*}
Since  $ \frac{\la}{\la-\omega_{r,m}}(\delta + c_m) \to \delta + c_m < 1$ and $\frac{C_m a_\rho}{\la-\omega_{r,m}}\to 0$ as $\la \to \infty$,  it follows that there exists $\la_0$ such that for all $\la>\la_0$
$$
\frac{C_m a_\rho}{\la-\omega_{r,m}} + \frac{\la}{\la-\omega_{r,m}}(\delta + c_m)<1.
$$
Therefore $B_{0,m}$ is a Miyadera perturbation of $T_{0,m}$, and the stated result follows.
\end{proof}
Under the conditions of Theorem \ref{Miy}, it follows that constants $C(m)$ and $\theta(m)$ exist such that
\begin{equation}
\|S_{G_{0,m}}(t)f\|_{[0,m]} \leq C(m) e^{\theta(m)t}\|f\|_{[0,m]}, \ \mbox{ for all } f \in X_{0,m} \mbox{ and } t \geq 0.
\label{Stype}
\end{equation}
Moreover, an alternative, but equivalent, representation of  the generator $G_{0,m}$ is
\begin{equation}
G_{0,m}:= T^0_{0,m} +  A^{(1)}_{0,m} + A_{0,m} +B_{0,m}   = T^0_{0,m} + A^{(1)}_{0,m} + F_{0,m},
\label{genrep}
\end{equation}
where $ A^{(1)}_{0,m}\,,A_{0,m}$ and $ B_{0,m}$ are defined by \eqref{defnA}, \eqref{defnA1} and \eqref{B0m} respectively, and
\begin{align*}
[T^0_{0,m}f](x) &:= -\p_x[r(x)f(x)]\,; \\ D(T^0_{0,m}) &:= \left\{f \in X_{0,m}\; :\; rf \in AC(\mbb R_+)\;{\rm
and}\;\frac{d}{dx}(rf)\in X_{0,m}\right\}.
\end{align*}
As with the operator $T_{0,m}$, the homogeneous boundary condition must also  be incorporated in the above definition of $D(T^0_{0,m})$ when \eqref{assr1} holds.  In \cite[Theorem 2.2]{Ban2020}, it is shown that  the fragmentation operator, $(F_{0,m},D(A_{0,m})):= (A_{0,m}+B_{0,m},D(A_{0,m}))$, is the generator of an analytic semigroup on $X_{0,m}$.

We now establish a regularising property of the growth-fragmentation semigroup that holds under an additional assumption on the fragmentation rate function $a$.   The proof involves the adjoint semigroup,
$\left(S^*_{G_{0,m}}(t)\right)_{t \geq 0}$, defined on the dual space $X_{0,m}^*$, where the latter can be identified with the function space
\[
L_{\infty,1/w_m}:= \left\{f : f \mbox{ is measurable on } \mbb{R}_+ \mbox{ and }  \|f\|_{\infty,m}:= \mbox{ess}\!\sup_{x \in \mbb{R}_+} \frac{|f(x)|}{w_m(x)} < \infty\right\}
\]
via the duality pairing
\[
\langle f,g \rangle := \cl{0}{\infty} f(x)g(x)\md x,\ \ f \in L_{\infty,1/w_m},\, g \in X_{0,m}.
\]
Since $w_m \in  L_{\infty,1/w_m}$, we can define
\begin{equation}\label{Psi}
\Psi_m(x,t) := [S^*_{G_{0,m}}(t)w_m](x),\ (x,t) \in \mathbb{R}_+^2.
\end{equation}
The following result is an extension of \cite[Lemma 2.7]{Ber2019} to the more general setting of this paper. The proof, while using the  better characterization of the generator obtained in Theorem \ref{Miy}, essentially follows the lines of \textit{op.cit}.
\begin{theorem}
In addition to the conditions required for Theorem \ref{Miy} to hold, assume that positive constants  $a_0\,,\gamma_0$  and $x_0$ exist such that
\begin{equation}
a(x)\geq a_0x^{\gamma_0},\quad \mbox{ for all } x\geq x_0.
\label{assa1}
\end{equation}
Then, for any $n,\, p$ and $m$ satisfying  $\max\{1,l\} < n < p < m$, there are constants $C=C(m,n,p)>0$ and  $\theta=\theta(m,n)>0$ such that
\begin{equation}
 \|S_{G_{0,p}}(t)\mr f\|_{[0,m]}  \leq Ce^{\theta t} t^{\frac{n-m}{\gamma_0}}\|f\|_{[0,p]},\ \mbox{ for all } f \in X_{0,p}.
 \label{regest}
 \end{equation}\label{regthm}
\end{theorem}
\begin{proof}
First we note that $X_{0,m} \hookrightarrow X_{0,p} \hookrightarrow X_{0,n}$, where $\hookrightarrow$ denotes a continuous embedding. Moreover, for each $j=n,p,m$, the operator $G_{0,j}$ generates a positive strongly continuous semigroup \sem{S_{G_{0,j}}} on $X_{0,j}.$  Suppose, initially, that
$\mr f \in D(T_{0,m})_+ = D(G_{0,m})_+$. Then, for all $t \geq 0$,
\[
f(\cdot,t) := [S_{G_{0,m}}(t)\mr f](\cdot)= [S_{G_{0,p}}(t)\mr f](\cdot)\in D(G_{0,m}) = D(T^0_{0,m})\cap D(A_{0,m}).
\]
Consequently, we can multiply \eqref{reml} by $w_m(x) = 1+x^m$ and then integrate term by term to obtain, as in \cite[Lemma 5.2.17]{BLL},
\begin{align}
\frac{d}{dt} M_{0,m}(t) &=  \cl{0}{\infty}\left(mr(x)x^{m-1}  -(N_0(x)+N_m(x))a(x)\right)f(x,t)  \mdm{d}x \nn\\
&\phantom{xx}- \cl{0}{\infty} a_1(x) f(x,t)w_m(x)\md x \leq \cl{0}{\infty} \Phi_m(x)f(x,t)\md x\,,
\label{subfuncta'}
\end{align}
where
\begin{equation}\label{Phi}
\Phi_m(x):= mr(x)x^{m-1}  -(N_0(x)+N_m(x))a(x),\ \  x \in \mathbb{R}_+.
\end{equation}
Recalling from \eqref{bmom} that $n_m(y) \leq c_my^m$ for all $y \geq y_m$, where $0 < c_m < 1$, we choose a positive constant $R_m > \max\{1,x_0,y_m\}$ such that
\begin{equation}
 (b_0(1+x^l) -1) -(1-c_m)x^m \leq 0, \ \mbox{ for all } x \geq R_m.
 \label{Rm}
 \end{equation}
It then follows from \eqref{fmlras}, \eqref{PhPr005}, \eqref{bmom} and \eqref{assa1} that, for any fixed $R \geq R_m$ and for all $x\geq R$, we have
\begin{align*}
\Phi_m(x)&\leq m\ti r (1+x)x^{m-1} +(b_0(1+x^l) -1) -(1-c_m)x^m)a_0 R^{\gamma_0}\\
& \leq (2m\ti r - (1-c_m)a_0R^{\gamma_0})w_m(x) + (b_0w_l(x) -c_m)a_0R^{\gamma_0}.
\end{align*}
If we now impose the further restriction that $R_m$ is also chosen so that
$$
2m\ti r - (1-c_m)a_0R^{\gamma_0}\leq -d_m R^{\gamma_0}, \ \mbox{ for each } R \geq R_m,
$$
where $d_m>0$, then, for any $x$ and $R$ satisfying $x \geq R \geq R_m$, we have
\begin{equation}
\Phi_m(x) \leq -d_mR^{\gamma_0}w_m(x) + b_0 a_0R^{\gamma_0}w_n(x).
\label{Phi1}
\end{equation}
Turning to the case when $x\leq R$, we have $N_m(x) \geq 0$ for all $x$, by \eqref{Nm}, and  know also that \eqref{Rm} holds for $x \in [R_m,R]$.  Consequently, on setting  $a_{R_m} = \mathrm{ess}\!\!\!\!\sup\limits_{x \in [0,R_m]}a(x)$, we obtain, for $0 < x \leq R$,
\begin{align*}
&\Phi_m(x) \leq 2 m\ti r w_m(x)  + (b_0(1+R_m^l) -1)a_{R_m} \\
&= -d_mR^{\gamma_0}w_m(x) +(d_mR^{\gamma_0}+ 2 m\ti r )w_m(x) + (b_0(1+R_m^l) -1)a_{R_m}\\
&\leq -d_mR^{\gamma_0}w_m(x) +\left((d_mR^{\gamma_0}+ 2 m\ti r )\frac{w_m(x)}{w_n(x)} + \frac{(b_0(1+R_m^l) -1)a_{R_m}}{w_n(x)}\right)w_n(x)\\
&\leq -d_mR^{\gamma_0}w_m(x)\! +\!\left(\!(d_mR^{\gamma_0}\!+\! 2 m\ti r )(1+R^{m-n})\! +\! \frac{(b_0(1\!+\!R_m^l) -1)a_{R_m}}{w_n(x)}\!\right)w_n(x),
\end{align*}
where we have used the inequality $w_m(x)/w_n(x) \leq 1 + x^{m-n},\, x > 0$. It follows that, for any fixed $R \geq R_m$, there exist positive constants $d_m$ and $D_m$ such that
$$
\Phi_m(x) \leq -d_m R^{\gamma_0} w_m(x) + D_m R^{\gamma_0+m-n}w_{n}(x),\ \mbox{ for all } x\in \mbb R_+,
$$
and therefore, from \eqref{subfuncta'},
\begin{equation}
\frac{d}{dt}M_{0,m}(t) \leq - d_m R^{\gamma_0} M_{0,m}(t) + D_m R^{\gamma_0+m-n} M_{0,n}(t).
\label{Mom1}
\end{equation}
Since  Theorem \ref{Miy} ensures that
\[
M_{0,n}(t)\! =\! \|S_{G_{0,m}}(t)\mr f\|_{[0,n]}\! =\! \|S_{G_{0,n}}(t)\mr f\|_{[0,n]}\!\leq\! C(n) e^{\theta(n) t}\|\mr f\|_{[0,n]}\! =:\! \sigma_n(t)\|\mr f\|_{[0,n]},
\]
\eqref{Mom1} leads to
$$
\frac{d}{dt}(e^{d_m R^{\gamma_0} t}M_{0,m}(t)) \leq  D_m C(n)R^{\gamma_0+m-n}e^{(d_mR^{\gamma_0} + \theta(n))t}\|\mr f\|_{[0,n]}.
$$
Hence, for any fixed $R\geq R_m$, and with $\Psi_m$ defined by \eqref{Psi},
\begin{equation}\label{Mom3}\begin{split}
M_{0,m}(t) &= \cl{0}{\infty} [S_{G_{0,m}}(t)\mr f](x)w_m(x) \md x =  \cl{0}{\infty} \mr f(x)[S^*_{G_{0,m}}(t)w_m](x) \md x\\& = \cl{0}{\infty}\Psi_m(x,t)\mr f(x) \md x \\
&\leq e^{-d_m R^{\gamma_0} t} \| \mr f\|_{[0,m]} + \frac{D_m R^{\gamma_0}}{d_mR^{\gamma_0} + \theta(n)}R^{m-n}(\sigma_n(t)-e^{-d_m R^{\gamma_0} t})\| \mr f\|_{[0,n]}\\
&\leq e^{-d_m R^{\gamma_0} t} \| \mr f\|_{[0,m]} + D'_mR^{m-n}\sigma_n(t)\| \mr f\|_{[0,n]}\\
&= \cl{0}{\infty}(e^{-d_m R^{\gamma_0} t} w_m(x) + D'_m R^{m-n} \sigma_n(t) w_n(x))\mr f(x) \md x.
\end{split}
\end{equation}
Since all positive $C^\infty_0(\mbb R_+)$ functions are in $D(G_{0,m})_+$, this leads to
\begin{equation}
\Psi_m(x,t) \leq e^{-d_m R^{\gamma_0} t} w_m(x) + D'_m R^{m-n} \sigma_n(t) w_n(x),
\label{Psi1}
\end{equation}
for almost any $x>0$ and each $R\geq R_m$.
Next, as in the proof of \cite[Lemma 2.7]{Ber2019},  we use the fact that $t, x$ and $R$ are independent. Consequently, for  fixed $t$ and $x$, with  $x \geq e^{\frac{d_m R_m^{\gamma_0} t}{m-n}}$, we can define $R \,(=R(x,t))$ by $R:= \left(\frac{m-n}{d_m}\frac{\log x}{t}\right)^{1/\gamma_0}$. It then follows from \eqref{Psi1} that
\begin{equation}
\label{Mom5}
\begin{split}
\Psi_m(x,t) &\leq x^{n-m}w_m(x) + D'_m \left(\frac{m-n}{d_m}\right)^{1/\gamma_0}t^{\frac{n-m}{\gamma_0}}(\log x)^{\frac{n-m}{\gamma_0}} \sigma_n(t) w_n(x)\\
&\leq D_{m,n,p}\, \widehat{\sigma}_n(t)t^{\frac{n-m}{\gamma_0}} w_{p}(x),
\end{split}
\end{equation}
 where $p$ is any number bigger than $n$, the function $\widehat{\sigma}_n(t)$ is bounded as $t\to 0^+$ and exponentially bounded as $t\to \infty$, and $D_{m,n,p}$ is a constant depending on $m,n,p$. For $x < e^{\frac{d_m R_m^{\gamma_0} t}{m-n}},$ we take $R=R_m$ and use the fact that $w_m(x)$ and $w_n(x)$ are increasing functions to obtain
 \begin{equation}
 \label{Mom6}
 \begin{split}
 \Psi_m(x,t) &\leq e^{-d_m R_m^{\gamma_0} t} w_m(x) + D'_m R_m^{m-n} \sigma_n(t) w_n(x)\\&  \leq e^{-d_m R_m^{\gamma_0} t} w_m\left(e^{\frac{d_m R_m^{\gamma_0} t}{m-n}}\right) + D'_m R_m^{m-n} \sigma_n(t) w_n\left(e^{\frac{d_m R_m^{\gamma_0} t}{m-n}}\right)\\
 &\leq D_{m,n} \widetilde{\sigma}_n(t) e^{\frac{md_m R_m^{\gamma_0} t}{m-n}} \leq D_{m,n}\widetilde{\sigma}_n(t) e^{\frac{md_m R_m^{\gamma_0} t}{m-n}} w_{p}(x).
 \end{split}
 \end{equation}
 Summarising, there are constants  $C=C(m,n,p)$ and $\theta = \theta(m,n)$ such that, for almost all $x>0$ and $t>0$,
 $$
 \Psi_m(x,t) \leq Ce^{\theta t} t^{\frac{n-m}{\gamma_0}} w_{p}(x)
 $$
 and hence, using \eqref{Mom3},
 $$
 \|S_{G_{0,p}}(t)\mr f\|_{[0,m]}  \leq Ce^{\theta t} t^{\frac{n-m}{\gamma_0}}\cl{0}{\infty} \mr f(x) w_{p}(x) \md x.
 $$
 The inequality can be extended to $\mr f \in X_{0,p}$ by linearity and density.
\end{proof}
\begin{corollary}
Let the assumptions of Theorem \ref{regthm} be satisfied. Then $S_{G_{0,p}}(t):D(G_{0,p}) \to D(G_{0,m})$ for all $t>0$ .\label{correg}
\end{corollary}
\begin{proof}
Let $m,n$ and $p$ be as in Theorem \ref{regthm}. Since $f$ and $G_{0,p}f$ belong to $ X_{0,p}$, both  $S_{G_{0,p}}(t)f$ and $S_{G_{0,p}}(t)G_{0,p}f$ are in $X_{0,m}$ for $t>0$, and therefore we can evaluate
$$\frac{S_{G_{0,m}}(h) - I}{h}S_{G_{0,p}}(t)f = \frac{S_{G_{0,p}}(h) - I}{h}S_{G_{0,p}}(t)f=S_{G_{0,p}}(t)\frac{S_{G_{0,p}}(h) - I}{h}f.$$
It then follows from Theorem \ref{regthm} that
\begin{align}
&\lim\limits_{h\to 0^+}\left\|\frac{S_{G_{0,m}}(h) - I}{h}S_{G_{0,p}}(t)f - S_{G_{0,p}}(t)G_{0,p} f\right\|_{[0,m]}\nn\\
&\phantom{xx}=\lim\limits_{h\to 0^+}\left\|S_{G_{0,p}}(t)\left(\frac{S_{G_{0,p}}(h) - I}{h}f - G_{0,p} f\right)\right\|_{[0,m]} \nn\\ &\phantom{xx} \leq \lim\limits_{h\to 0^+}Ce^{\theta t}t^{-\frac{m-n}{\gamma_0}}\left \|\frac{S_{G_{0,p}}(h) - I}{h}f - G_{0,p} f\right\|_{[0,p]}
=0,\label{247}
\end{align}
which establishes that $S_{G_{0,p}}(t)f \in D(G_{0,m})$ for all $t > 0$.
\end{proof}
\begin{corollary} Assume that   \eqref{fmlras}, \eqref{PhPr005}, \eqref{goodchar}, \eqref{massump} and \eqref{assa1} are all satisfied, and let $p>\max\{1,l\}$. Then,  for each
	$\mr f \in X_{0,m}\cap D(G_{0,p}),$   problem \eqref{reml} has a classical solution in $X_{0,m}$.
\end{corollary}
\begin{proof}
	Let $f(t) = S_{G_{0,m}}(t)\mr f$. We can assume that $p< m$, as otherwise $\mr f \in D(G_{0,m})$. Then, for all $t>0$, $S_{G_{0,p}}(t)\mr f=S_{G_{0,m}}(t)\mr f$ and, by Corollary \ref{correg}, $S_{G_{0,p}}(t)\mr f\in D(G_{0,m})$ so that, as in \eqref{247},
	\begin{align*}
\lim\limits_{h\to 0^+} \!\frac{f(t+h)\! -\! f(t)}{h} \! &=\!	\lim\limits_{h\to 0^+}\!\frac{S_{G_{0,m}} (h)\! -\! I}{h} S_{G_{0,m}} (t)\mr f =\!
\lim\limits_{h\to 0}\frac{S_{G_{0,p}} (h)\! -\! I}{h} S_{G_{0,p}}(t)\mr f \\
& =  G_{0,p} S_{G_{0,p}} (t)\mr f = G_{0,m} S_{G_{0,m}} (t)\mr f
\end{align*}
in $ X_{0,m},$ where the last equality follows from  Corollary \ref{correg}.
\end{proof}

\section{Coagulation-fragmentation with growth}

The results obtained in the previous section can now be exploited to establish the well-posedness of the initial value problem (IVP)
\eqref{initprof}. The restrictions placed on $r,a$ and $b$ for Theorem \ref{regthm} to hold continue to be assumed, but now we specify that $a_1(x) := \beta(1+x^\alpha)$, where $\beta$ is a constant that will be determined later, and   $0 < \alpha < \gamma_0$, so that, due to \eqref{assa1}, $a_1(x)/a(x)$ remains bounded as $x \to \infty$. The coagulation kernel is required to satisfy
\begin{equation}
k(x,y) \leq k_0 (1+ x^\alpha)(1+y^\alpha),
\label{kass}
\end{equation}
for some positive constant  $k_0$.  It is convenient to express \eqref{initprof} in the form
\begin{eqnarray}
 \partial_t f(x,t)  &=& - \p_x[r(x)f(x,t)] - \beta(1+x^\alpha)f(x,t) + {\mathfrak F}f(x,t)\nn\\
 &&\phantom{xx} +
\mathfrak{K}_\beta f(x,t)\ , \  (x,t)\in \mbb{R}_+^2, \label{rem2} \\
 f(x,0)  &=&  \mr f(x)\ , \ \  x\in \mathbb{R}_+\ , \label{rem3}
\end{eqnarray}
where
\begin{equation}\label{modifiedcoag}
\mathfrak{K}_\beta f(x,t):= \beta(1+x^\alpha)f(x,t) + \mathfrak{K}f(x,t),
\end{equation}
and $\mathfrak{F},\,\mathfrak{K}$ are given by \eqref{wlcontsfrag} and \eqref{wlcontscoag} respectively. Denoting the generator of the growth-fragmentation semigroup in this case by $G_{0,m}^{(\beta)}$, the corresponding abstract formulation of the IVP \eqref{rem2} - \eqref{rem3} can be written as
\begin{equation}\label{ACPbeta}
\frac{d}{dt}f(t) = G_{0,m}^{(\beta)}f(t) + K_{0,m}^{(\beta)}f(t),\ t > 0; \ \ f(0) = \mr f,
\end{equation}
where the operator $K_{0,m}^{(\beta)}$  is defined on $X_{0,m}$ via
\begin{eqnarray}
(K_{0,m}^{(\beta)}f)(x) &:=& \beta(1+x^\alpha)f(x) + \frac{1}{2}\,\int_0^x
k(x-y,y)f(x-y)f(y)\,\md y \nn \\
&& \quad - f(x)\,\int_0^{\infty}
k(x,y)f(y)\,\md y,\ x \in \mbb{R}_+. \label{coagop}
\end{eqnarray}

The following inequalities will often be used . For $0\leq \delta \leq \eta$ and $x\geq 0$
\begin{equation}
(1+x^\delta)\leq 2(1+x^\eta)
\label{in1}
\end{equation}
and, for $\delta,\eta\geq 0$ and $x\geq 0$,
\begin{equation}
(1+x^\delta)(1+x^\eta)\leq 4(1+x^{\delta+\eta}).
\label{in2}
\end{equation}

\subsection{Local Existence}

We begin by proving the local (in time) existence and uniqueness  of a mild solution to \eqref{ACPbeta}.

\begin{theorem}\label{lm3.1}
Let $r, a , b$ satisfy the conditions for Theorem \ref{regthm} to hold, and define $a_1(x) := \beta(1+x^\alpha)$, where $\beta > 0$ is appropriately chosen and $ 0 < \alpha < \gamma_0$, with $\gamma_0$  the constant given in \eqref{assa1}. Further, let $k$ satisfy \eqref{kass} and let $m > \alpha+\max\{1,l\}$.    Then, for each 	$\mr f \in X_{0,m,+}$, the semilinear ACP \eqref{ACPbeta} has a unique nonnegative mild solution $f \in C([0, \tau_{\max}), X_{0,m})$ defined on its maximal interval of existence $[0,\tau_{\max})$, where $\tau_{max} = \tau_{\max}(\mr f)$. If $\tau_{\max}< \infty$, then $\|f(t)\|_{[0,m]}$ is unbounded as $t\to \tau_{max}^-$.
\end{theorem}
\begin{proof}
 Let $p$ be defined by
\begin{equation}
p:= m-\alpha,\label{p}
\end{equation}
and, noting that  $(m-n)/\gamma_0 = (p-n)/\gamma_0 + \alpha/\gamma_0$ , we are then able to choose $n < p$ such that $(m-n)/\gamma_0 < 1$ and $n>\max\{1,l\}$, so that $m,n,p$ satisfy assumptions of Theorem \ref{regthm}.  We begin by showing that the bilinear form $\mc K_{0,m}^{(\beta)}$, defined by
\begin{align}
[\mc K_{0,m}^{(\beta)}(f,g)](x) &:= \beta(1 +x^\alpha)f(x)- f(x)\!\!\cl{0}{\infty}\!\!k(x,y)g(y)\md y\nn\\
 &\phantom{xx}+
\frac{1}{2}\cl{0}{x}\!\!k(x-y,y)f(x-y)g(y)\md y,\label{mcK}
\end{align}
is a continuous mapping from $X_{0,m} \times X_{0,m}$ into $X_{0,p}$. From \eqref{kass}, \eqref{in2} and \eqref{p}, we obtain, for all $f,g \in X_{0,m}$,
\begin{equation}\label{wllinear}
\beta \cl{0}{\infty}(1+x^\alpha)|f(x)|w_p(x) \md x  \leq  4\beta \|f\|_{[0,m]},
\end{equation}
\begin{equation}\label{cest1}
\int_{0}^\infty  |f(x)|\left(\int_{0}^\infty
k(x,y) |g(y)|\md y\right)w_p(x)\md x  \leq  4k_0 \|f\|_{[0,m]}\|g\|_{[0,m]}
\end{equation}
and, in a similar way,
\begin{equation}\label{cest2}\begin{split}
&\frac{1}{2}\int_{0}^{\infty}\!\!\! \left(\int_{0}^{x}
k(x-y,y) |f(y)||g(x-y)|\md y\right)w_{p}(x)\md x\\&\phantom{xxx}= \frac{1}{2}\int_{0}^{\infty}\!\!\int_{0}^{\infty}\!\! k(x,y)|f(y)||g(x)|w_{p}(x+y)\md x\md y\\
&\phantom{xxx}\leq  2^{p-1}k_0\cl{0}{\infty}\cl{0}{\infty} |f(y)||g(x)| ((1+x^\alpha)(1+y^\alpha) (w_{p}(x)+w_{p}(y))\md x\md y \\
&\phantom{xxx}\leq 2^{p-1} k_0(4\|g\|_{[0,m]} \|f\|_{[0,\alpha]} + 4\|g\|_{[0,\alpha]}\|f\|_{[0,m]})\leq 2^{p+3} k_0\|f\|_{[0,m]}\|g\|_{[0,m]}.\end{split}
\end{equation}
Hence,
\begin{equation}\label{boundedbl}
\|\mc K_{0,m}^{(\beta)}(f,g) \|_{[0,p]} \leq  \left(4\beta +(4+2^{p+3})k_0\|g\|_{[0,m]}\right)\|f\|_{[0,m]},
\end{equation}
{ for all } $f,g \in X_{0,m}$. Since $K_{0,m}^{(\beta)}f = \mc K_{0,m}^{(\beta)}(f,f)$, it follows that $K_{0,m}^{(\beta)}$ is a continuous mapping from $X_{0,m}$ into $X_{0,p}$. Consequently, the integral equation that arises as the mild formulation of \eqref{ACPbeta} can be written as
 \begin{equation}\label{eq3.2}
	f(t) =  S_{G_{0,m}^{(\beta)}}(t) \mr f +  \int_{0}^{t} S_{G_{0,p}^{(\beta)}} (t - s) K_{0,m}^{(\beta)} f(s) \md s.
\end{equation}
Next consider the set
\begin{equation}
\mc U := \{ f \in X_{0,m,+}:\;\| f\|_{[0,m]}\leq 1+b\},
\label{ball}
\end{equation}
for some arbitrarily fixed $b>0$. For each $ f\in \mc U$,  we can use \eqref{kass}, \eqref{in2} and the fact that $\alpha < \gamma_0 \leq m$, to obtain
$$
\int_{0}^{\infty} k(x,y)f(y)\md y \leq 2k_0 (1 + x^\alpha)\|f\|_{[0,m]} \leq \beta(1+x^\alpha), \ \mbox{ for all } x > 0,
$$
where we now define  \begin{equation}
\beta := 2k_0(1 + b),
 \label{gammak}\end{equation}
and therefore, with this choice of $\beta$,
\begin{equation}\label{Cpos}
(K_{0,m}^{(\beta)} f)(x) \geq \frac{1}{2}\cl{0}{x}k(x-y,y) f(x-y)f(y)\md y \geq 0.
\end{equation}
Also, from \eqref{boundedbl}, we have
\begin{equation}
\|K_{0,m}^{(\beta)} f\|_{[0,p]} \leq   K(\mc U), \ \mbox{ for all } f \in \mc U,
\label{cest2a}
\end{equation}
where $K(\mc U) = \frac{\beta^2}{k_0}(2 +(1+2^{p+1}))$, and
\begin{align}
&\Vert K_{0,m}^{(\beta)} f - K_{0,m}^{(\beta)} g\Vert_{[0,p]} \nn\\
 &\leq  4\beta \Vert f-g\Vert_{[0,m]} + (4+ 2^{p+3})k_0 \left( \Vert f \Vert_{[0,m]}+ \Vert g \Vert_{[0,m]}\right)\Vert f - g \Vert_{[0,m]}\nn\\&\leq  L(\mathcal{U})\Vert f - g \Vert_{[0,m]},  \label{K1B}
\end{align}
 for all  $f,g \in \mc U$, where, by \eqref{gammak}, $
L(\mathcal{U}) = 8\beta(1 +  2^{p}).
$

For $\mr f \in X_{0,m,+}$ satisfying
\begin{equation}\label{wlfin}
\|\mr f\|_{0,m} \leq { b},
\end{equation}
we define the operator
\begin{equation}\label{wlSL}
\textsl{T}f(t) = S_{G_{0,m}^{(\beta)}}(t)\mr f +
\int_{0}^{t}S_{G_{0,p}^{(\beta)}}(t-s)K_{0,m}^{(\beta)} f(s)\md s
\end{equation}
 in the space $Y_m=C([0,\tau], \mc U),$ with $\mc U$ defined by (\ref{ball}) and $\tau$ to be determined so that $\textsl{T}$ is a contraction on $Y_m$, when $Y_m$ is  equipped with the metric induced by the norm from $C([0,\tau], X_{0,m})$.  First, observe that $\textsl{T}f \in C([0,\tau], X_{0,m,+})$ for all $f \in Y_m$.  Indeed, for any $t\geq 0$ and $h>0$, with $t + h \leq \tau$,
 \begin{align*}
 &\|\textsl{T}f(t+h) - \textsl{T}f(t)\|_{[0,m]}\\
  = &\left\|\int_{0}^{t+h}S_{G_{0,p}^{(\beta)}}(t+h-s)K_{0,m}^{(\beta)} f(s)\md s-\int_{0}^{t}S_{G_{0,p}^{(\beta)}}(t-s)K_{0,m}^{(\beta)} f(s)\md s\right\|_{[0,m]}\\
 \leq & \int_{t}^{t+h}\left\|S_{G_{0,p}^{(\beta)}}(t+h-s)K_{0,m}^{(\beta)} f(s)\right\|_{[0,m]}\md s \\
 &+ \int_{0}^{t}\left\|S_{G_{0,p}^{(\beta)}}(t-s)(S_{G_{0,p}^{(\beta)}}(h)-I)K_{0,m}^{(\beta)} f(s)\right\|_{[0,m]}\md s=   I_1(h)+I_2(h).
 \end{align*}
 Now, by \eqref{regest} and \eqref{cest2a},
 $$
 \|S_{G^{(\beta)}_{0,p}}(t+h-s)K_{m,\beta} f(s)\|_{[0,m]}\leq C(m,n,p)e^{\theta(m,n) (t+h-s)} (t+h-s)^{\frac{n-m}{\gamma_0}} K(\mc U).
 $$
 Since $(n-m)/\gamma_0 > -1$, it follows that
 $$
 \cl{t}{t+h}(t+h-s)^{\frac{n-m}{\gamma_0}} ds = \cl{0}{h} \sigma^{\frac{n-m}{\gamma_0}} d\sigma  \to 0 \ \mbox{ as } h \to 0^+,
 $$
and therefore $\lim_{h\to 0^+} I_1(h) = 0$.
 Similarly,
 \begin{align}
 &\|S_{G_{0,p}^{(\beta)}}(t-s)(S_{G_{0,p}^{(\beta)}}(h)-I)K_{0,m}^{(\beta)} f(s)\|_{[0,m]} \nn \\
 \leq & C(m,n,p)e^{\theta(m,n) (t-s)} (t-s)^{\frac{n-m}{\gamma_0}}\|(S_{G_{0,p}^{(\beta)}}(h)-I)K_{0,m}^{(\beta)} f(s)\|_{[0,p]}\nn\\
 &\phantom{xx}\leq  C(m,n,p)e^{\theta(m,n) (t-s)} (t-s)^{\frac{n-m}{\gamma_0}}(C(p)
  e^{\theta(p)h}+1)\|K_{m,\beta} f(s)\|_{[0,p]}\nn\\
 & \phantom{xx}\leq C(m,n,p,\tau)(t-s)^{\frac{n-m}{\gamma_0}},\label{osz1}
 \end{align}
 where  $C(m,n,p,\tau)$ is a constant that is independent of $h$. Thus, from the second line in the above calculation,  we see that the integrand in $I_2(h)$ converges to zero as $h \to 0^+$  for each $0\leq s< t$ and the last line ascertains that this convergence is dominated by an integrable function. Hence, an  application of the Lebesgue dominated convergence theorem shows that $\textsl{T}f$ is right continuous at $t$ for all $t \in [0,\tau)$.  When $0<h<t \leq \tau$ and $t-h \geq 0$ , we have
 \begin{align*}
 &\|\textsl{T}f(t-h) - \textsl{T}f(t)\|_{[0,m]} \\
  = &\left\|\int_{0}^{t-h}S_{G_{0,p}^{(\beta)}}(t-h-s)K_{0,m}^{(\beta)} f(s)\md s-\int_{0}^{t}S_{G_{0,p}^{(\beta)}}(t-s)K_{0,m}^{(\beta)} f(s)\md s\right\|_{[0,m]}\\
 \leq &\int_{t-h}^{t}\left\|S_{G_{0,p}^{(\beta)}}(t-s)K_{0,m}^{(\beta)} f(s)\right\|_{[0,m]}\md s \\
 &+ \int_{0}^{t-h}\left\|S_{G_{0,p}^{(\beta)}}(t-h-s)(I-S_{G_{0,p}^{(\beta)}}(h))K_{0,m}^{(\beta)} f(s)\right\|_{[0,m]} \md s= I'_1(h)+I'_2(h).
 \end{align*}
 Arguing as before, we obtain $\lim_{h\to 0^+}I'_1(h) =0$.  As for $I'_2(h),$ we rewrite it as
\begin{align}
I'_2(h) &= \int_{h}^{t}\left\|S_{G_{0,p}^{(\beta)}}(t-\sigma)(I-S_{G_{0,p}^{(\beta)}}(h))K_{0,m}^{(\beta)} f(\sigma - h)\right\|_{[0,m]}\md \sigma\nn\\
&= \int_{0}^{t}\chi_{[h,t]}\left\|S_{G_{0,p}^{(\beta)}}(t-\sigma)(I-S_{G_{0,p}^{(\beta)}}(h))K_{0,m}^{(\beta)} f(\sigma - h)\right\|_{[0,m]} \md \sigma,\label{I'}
 \end{align}
 where $\chi_\Omega$ is the characteristic function of $\Omega$. Since $t \to K_{0,m}^{(\beta)} f(t)$ is a continuous function in $X_{0,p},$  $\lim\limits_{h\to 0^+} K_{0,m}^{(\beta)} f(\sigma - h) = K_{0,m}^{(\beta)} f(\sigma)$ for each $\sigma > 0.$ Then, on account of the local uniform boundedness of \sem{S_{G_{0,p}^{(\beta)}}}, a corollary of the Banach-Steinhaus theorem ensures that $\lim\limits_{h\to 0^+} (I-S_{G_{0,p}^{(\beta)}}(h))K_{0,m}^{(\beta)} f(\sigma - h) = 0$ for any fixed  $\sigma>0$ and we see that the integrand in \eqref{I'} converges to zero on $[0,t].$ Moreover, from  \eqref{osz1},
 \begin{align*}
 \|S_{G_{0,p}^{(\beta)}}(t-\sigma)(I-S_{G_{0,p}^{(\beta)}})(h))K_{0,m}^{(\beta)} f(\sigma-h)\|_{[0,m]}&\leq
   C(m,n,p,\tau)(t-\sigma)^{\frac{n-m}{\gamma_0}}
 \end{align*}
  for all $\sigma \in [h,t]$,  where, by \eqref{cest2a}, $K_{0,m}^{(\beta)} f(\sigma-h)$ is estimated by the $Y_m$ norm of $f$, and this is independent of $h$. Consequently, on applying the Lebesgue dominated convergence theorem once again, we obtain $\lim_{h\to 0^+}I'_2(h) =0$. Further, thanks to \eqref{gammak},
$\textsl{T}f(t) \geq 0$ since $f(s)\geq 0$ for all $s \in [0,\tau]$.

By continuity,  for any initial condition $\mr f$ satisfying \eqref{wlfin} we can choose $\tau'_1$  such that $\|\textsl{T}f(t)\|_{[0,m]}  \leq 1+b$ for $0\leq t \leq \tau'$. We need however, a more uniform estimate. For this,  using \eqref{regest}, we get
\begin{equation}\label{inter}
\begin{split}
\|\textsl{T}f(t)\|_{[0,m]}&\leq \|S_{G_{0,m}^{(\beta)}}(t)\mr f\|_{[0,m]} + \int_{0}^{t}\|S_{G_{0,p}^{(\beta)}}(t-s)K_{0,m}^{(\beta)}f(s)\|_{[0,m]}\md s\\
&\leq \|S_{G_{0,m}^{(\beta)}}(t)\mr f\|_{[0,m]} + C(m,n,p)e^{\theta(m,n)\tau} K(\mc U)\int_0^t (t-s)^{-\frac{m-n}{\gamma_0}} \, \rm{d}s \\
&\leq \|S_{G_{0,m}^{(\beta)}}(t)\mr f\|_{[0,m]}+ C(m,n,p)e^{\theta(m,n)\tau}  K(\mc U)\frac{\gamma_0}{\gamma_0+n-m} \tau^{\frac{\gamma_0+n-m}{\gamma_0}}\\
&\leq  1+b,
\end{split}
\end{equation}
  provided that $f(s) \in \mathcal{U}$ for $0 \leq s \leq \tau_{1}(\mc U),$  where the existence of such a $\tau_{1}$ is ensured by the fact that
  \begin{equation}
 \lim\limits_{\tau\to 0^+}\|S_{G_{0,m}^{(\beta)}}(\tau)\mr f\|_{[0,m]} =\|\mr f\|_{[0,m]} \quad \text{and}\quad  \lim\limits_{\tau\to 0^+}e^{\theta(m,n)\tau}   \tau^{\frac{\gamma_0+n-m}{\gamma_0}}=0.
\label{lims}
\end{equation}
Finally, to  establish that  $\textsl{T}$ is a contraction on $Y_m$  when  $\tau$ is sufficiently small, we use (\ref{K1B}) to obtain
\begin{eqnarray*}
&&\|\textsl{T}f(t)-\textsl{T}g(t) \|_{[0,m]} \leq   \cl{0}{t}\|S_{G_{0,p}^{(\beta)}}(t-s)([K_{0,m}^{(\beta)} f](s)- [K_{0,m}^{(\beta)}] g(s))\|_{m}\md s \\
&&\phantom{xx}\leq  C(m,n,p)e^{\theta(m,n)\tau} L(\mc U)\sup_{0\leq s\leq \tau}\| f(s)- g(s)\|_{[0,m]} \cl{0}{t}(t-s)^{-\frac{m-n}{\gamma_0}} \md s\\
&&\phantom{xx}\leq  C(m,n,p)e^{\theta(m,n)\tau}  L(\mc U)\frac{\gamma_0}{\gamma_0+n-m} \tau^{\frac{\gamma_0+n-m}{\gamma_0}}\sup_{0\leq s\leq \tau}\| f(s)-g(s)\|_{[0,m]}.
\end{eqnarray*}
 We now choose $\tau_{2}$ such that
 $$
 C(m,n,p)e^{\theta(m,n)\tau}  L(\mc U)\frac{\gamma_0}{\gamma_0+n-m} \tau^{\frac{\gamma_0+n-m}{\gamma_0}}<1,
 $$
 for any $0\leq \tau\leq \tau_2$. Hence, by taking $0<\tau = \min\{\tau_1,\tau_2\}$ we see that $\textsl{T}$ is a contractive mapping on $Y_m$. We note that $\tau$ is uniform on bounded subsets of $X_{0,m}$. Hence, in the usual way, we can extend the solution to the maximal interval $[0,\tau_{\max})$. The last statement of the theorem follows from the preceding observation that the length of the interval of existence  is uniform on bounded subsets and thus if the solution is bounded in any left neighbourhood of $\tau_{\max}$, it can be extended beyond it. \end{proof}
 The next objective is to prove that the mild solution of the previous theorem is, in fact, a classical solution of  \eqref{ACPbeta} under an additional restriction on $\mr f$. We require the following three lemmas.

\begin{lemma}
$K_{0,m}^{(\beta)} :X_{0,m} \to X_{0,p}$ is continuously Fr\'{e}chet differentiable. \label{lemdif}
\end{lemma}
\begin{proof}
Recall that $K_{0,m}^{(\beta)} f= \mc K_{0,m}^{(\beta)}(f,f)$, where $\mc K_{0,m}^{(\beta)}: X_{0,m}\times X_{0,m} \mapsto X_{0,p}$ is the bilinear form defined by \eqref{mcK}.
Using \eqref{cest1} and \eqref{cest2}, we see that
 $K_{0,m}^{(\beta)}$ is Fr\'{e}chet differentiable at each $f \in  X_{0,m}$, with Fr\'{e}chet derivative given by
\[
[\p K_{0,m}^{(\beta)} f]h := \beta w_{\alpha} h + \mc K_{0,m}^{(0)}(f,h) + \mc K_{0,m}^{(0)}(h,f)\,, \quad h \in X_{0,m}.
\]
Moreover, again by \eqref{cest1} and \eqref{cest2}, for any $f,g,h \in X_{0,m}$,
\begin{eqnarray*}
\Vert [\p K_{0,m}^{(\beta)} f]h - [\p K_{0,m}^{(\beta)} g]h\Vert_{[0,p]}
&=& \Vert \mc K_{0,m}^{(0)}(f-g,h) + \mc K_{0,m}^{(0)}(h,f-g) \Vert_{[0,p]} \\
&\leq& 8\beta(1 +  2^{p}) \Vert h \Vert_{[0,m]}\,\Vert f - g \Vert_{[0,m]} \to 0,
\end{eqnarray*}
 as $\Vert f - g \Vert_{[0,m]} \to 0$, uniformly in $\|h\|_{[0,m]}\leq 1$. Hence, the Fr\'{e}chet derivative is continuous.
\end{proof}
\begin{lemma}
Assume that $1<p<m$ and $0<T<\infty$ be arbitrary. Let $\sem{S}$ be a strongly continuous semigroup on $X_{0,p}$ and  $\{P(t)\}_{t\geq 0}$ be a family of bounded linear operators from $X_{0,m}$ to $X_{0,p}$ such that, for all $u \in X_{0,p}$ and $f \in X_{0,m}$,
\[
\|S(t)u\|_{[0,m]}\leq M(t)t^{-\kappa}\|u\|_{[0,p]}\,, \quad \|P(t)f\|_{[0,p]}\leq L(t)\|f\|_{[0,m]},
\]
where $M,L \in  L_{\infty}([0,T])$  and $0<\kappa<1$. Moreover, let $g \in C((0,T], X_{0,m})$ be such that
$\|g(t)\|_{[0,m]} \leq  G(t)t^{-\delta}$, where $G \in  L_{\infty}([0,T])$ and $0<\delta<1$.  Then the integral equation
\begin{equation}
f(t) = g(t) + \cl{0}{t}S(t-s)P(s)f(s)\md s
\label{inteq}
\end{equation}
has a unique solution $f$ in $C((0,T],X_{0,m})$ that satisfies $\|f(t)\|_{[0,m]}\leq F(t)t^{-\delta}$ for some function  $F\in L_{\infty}([0,T]).$\label{leminteq}
\end{lemma}
\begin{proof}We use some ideas from \cite[Lemma 3.2]{Banasiak2019}. Denoting Laplace convolution by  $\ast$  and  defining $\theta_{r}(t) := t^{-r}$, a simple argument shows  that $\theta_\delta\ast \theta_{\kappa}$ exists for any choice of $\delta < 1$ and $\kappa < 1$, with
\begin{equation}\label{WLconv}
(\theta_{\delta}\ast \theta_{\kappa})(t)  = B(1-\delta,1-\kappa)\,t^{1-\delta-\kappa} = B(1-\delta,1-\kappa)\,\theta_{\delta + \kappa - 1}(t),
\end{equation}
where $B$ is the Beta function. Let us denote by $M_T,L_T$ and $G_T$ the (essential) suprema  of the respective functions $M, L$ and $G$ on $[0,T]$.  Then, for all $t \in  (0,T]$,  we have
\begin{eqnarray*}
\|g_1(t)\|_{[0,m]} &:=& \left\|\cl{0}{t}S(t-s)P(s)g(s)\md s\right\|_{[0,m]} \leq M_T L_T G_T \cl{0}{t} (t-s)^{-\kappa} s^{-\delta}\md s \\&=& M_T L_T G_T B(1-\kappa, 1-\delta)t^{-(\delta+\kappa-1)},
\end{eqnarray*}
and, by induction,
\begin{eqnarray*}
\|g_n(t)\|_{[0,m]} &:=&\left\|\cl{0}{t}S(t-s)P(s)g_{n-1}(s)\md s\right\|_{[0,m]} \\
&\leq& (M_T L_T )^nG_T \prod\limits_{i=1}^{n}B(1-\kappa, i-(i-1)\kappa -\delta)t^{-(n\kappa -n +\delta)},
\end{eqnarray*}
for all $n \in \mathbb{N}$. Since $n -n\kappa -\delta = n(1 - \kappa - \delta/n)$,   there exists $n_0 \in \mathbb{N}$ such that $n -n\kappa -\delta  > 0$ for all $n \geq n_0$.  Then, denoting $g(t) =g_0(t)$,  we can re-write \eqref{inteq} as
\begin{equation}
f(t) - \sum\limits_{i=0}^{n_0-1} g_i(t) = g_{n_0}(t)+\cl{0}{t} S(t-s)\left(f(s) - \sum\limits_{i=0}^{n_0-1} g_i(s)\right)\md s
\label{ineq1}
\end{equation}
where  we have  $g_{n_0} \in C([0,T], X_{0,m})$.  Consider now an operator  on $C([0,T], X_{0,m})$ given by the formula
$$
Qu(t) = g_{n_0}(t)+\cl{0}{t} S(t-s)P(s)u(s)\md s.
$$
The argument used in Theorem \ref{lm3.1} to prove  the continuity of the operator $\textsl{T}$ can be applied again to establish the continuity of $Q$.  Then, for $u,v \in C([0,T], X_{0,m})$ we obtain
\begin{align*}
\|Qu(t) -Qv(t)\|_{[0,m]} &\leq M_TL_T \sup\limits_{s\in [0,T]}\|u(s)-v(s)\|_{[0,m]}\cl{0}{t} (t-s)^{-\kappa}\md s \\&= M_TL_T \sup\limits_{s\in [0,T]}\|u(s)-v(s)\|_{[0,m]}B(1-\kappa,1)t^{1-\kappa}
\end{align*}
and, again by induction,
\begin{align*}
&\|Q^ku(t) -Q^kv(t)\|_{[0,m]}\\ &\leq M^k_TL^k_T \sup\limits_{s\in [0,T]}\|u(s)-v(s)\|_{[0,m]}\prod_{i=0}^{k-1} B(1-\kappa, i+1 -i\kappa) t^{-i(\kappa-1)}.
\end{align*}
Now, using the fact that $B(x,y) = \frac{\Gamma(x)\Gamma(y)}{\Gamma(x+y)}$ and   \cite[Inequality 6.1.47]{abramowitz1964handbook}
$$
\frac{\Gamma(y)}{\Gamma(x+y)} \leq c_y y^{-x}, x> 0, y\to \infty,
$$
we see, on account of $ i+1 -i\kappa = i(1-\kappa) + 1\geq 1$ for $i\geq 0$, that
\begin{align*}
\prod_{i=0}^{k-1} B(1-\kappa, i+1 -i\kappa)&\leq \Gamma^k(1-\kappa)c_{1-\kappa}^k \prod_{i=0}^{k-1} (i(1-\kappa)+1)^{-(1-\kappa)}\\
&\leq \Gamma^k(1-\kappa)c_{1-\kappa}^k \left(\frac{1}{1-\kappa}\right)^{k(1-\kappa)}\left(\frac{1}{(k-1)!}\right)^{1-\kappa}.
\end{align*}
Hence, for some constant $C_T,$
\begin{align*}
\sup\limits_{t\in [0,T]}\|Q^ku(t) -Q^kv(t)\|_{[0,m]} &\leq \left(\frac{C_T^{\frac{k}{1-\kappa}}}{(k-1)!}\right)^{1-\kappa} \sup\limits_{s\in [0,T]}\|u(s)-v(s)\|_{[0,m]}
\end{align*}
and therefore there exists $k$ such that $Q^k$ is a contraction.  Thus, the equation $u = Qu$ has  a unique  solution $u \in C([0,T], X_{0,m})$ (the uniqueness follows from the Gronwall-Henry inequality, see \cite[Lemma 3.2]{Banasiak2019}) and therefore
$$
f(t) = u(t) + \sum\limits_{i=0}^{n_0-1} g_i(t)
$$
is a unique solution to \eqref{inteq} satisfying the growth condition at $t=0$.
\end{proof}
We now give the following lemma which seems to belong to mathematical folklore.
\begin{lemma}
Let $X,Y$ be Banach spaces and let $K$ be a continuously Fr\'{e}chet differentiable operator from $X$ to $Y$ with Fr\'{e}chet derivative $\p K\in C( X, \mc L(X,Y))$. Further, let the remainder $\omega: X\times X\mapsto Y$ be defined by
$$
K(x+h)- K(x) - \p K(x) h = \omega(h,x), \quad x, h\in X.
$$
Then the function $\omega_0(h,x) := \frac{\omega(h,x)}{\|h\|_X}$ if $h\neq 0$ and $\omega_0(0,x) := 0$ otherwise, is continuous.\label{rem}\end{lemma}
  \begin{proof}
  By the definition of Fr\'{e}chet differentiability, $\lim_{h\to 0}\frac{\omega(h,x)}{\|h\|_X} = 0$ for any $x\in X$.
  The only questionable points are $(h,x) =(0,x)$. Let us consider $(h_n,x_n)\to (0,x)$,  where $h_n\neq 0$ for  $n \in \mathbb{N}$. Then
\begin{align*}
\|\omega_0(h_n,x_n)\|_Y& = \frac{\left\|K(x_n+h_n)- K(x_n) - \p K(x_n) h_n\right\|_Y }{\|h_n\|_X}\\
& \leq  \cl{0}{1}\|\p K(x_n + th_n) - \p K(x_n) \| \md t.
\end{align*}
Now, $ \|\p K(x_n + th_n) - \p K(x_n) \| \to 0$ for any $t\in [0,1]$   and $\|\p K(x_n + th_n) - \p K(x_n) \|$ is bounded irrespective of $n$ as $\p K$ is continuous and the sets $\{x_n \}_{n \in \mbb N}$ and $\{x_n + th_n\}_{n\in \mbb N, t\in [0,1]}$  are compact. Consequently, $\|\omega_0(h_n,x_n)\|_Y\to 0$ by the Lebesgue dominated convergence theorem.  \end{proof}
In the next theorem we address the issue of differentiability of the mild solution constructed in Theorem \ref{lm3.1} and it being a classical solution to \eqref{ACPbeta}. The result is similar to that for analytic semigroups in that the mild solution in a smaller space (here $X_{0,m}$) is a classical solution in a bigger space (here $X_{0,p}$), see \cite[Definitions 7.0.1 \& 7.0.2]{Lun} or \cite[Section 4.7.1]{SY}.
\begin{theorem}\label{th3.4} Let the assumptions of Theorem \ref{lm3.1} hold and assume also that $\mr f \in X_{0,m} \cap D(G_{0,p}^{(\beta)})$, where $p = m-\alpha$.  Then the mild solution $f$,  defined on its maximal interval of existence $[0,\tau_{\max})$, satisfies  $f \in C([0, \tau_{\max}), X_{0,m}) \cap C^{1}((0, \tau_{\max}), X_{0,m})\cap C((0, \tau_{\max}), D(G_{0,p}^{(\beta)}))$ and is a classical solution to \eqref{ACPbeta} in $X_{0,p}$.\end{theorem}
\begin{proof} The proof follows the lines of \cite[Theorem 6.1.5]{Pa} but additional steps are required due to the unboundedness of the nonlinear term. To simplify the notation we observe that it suffices to prove the additional regularity on $(0,\tau)$ of the local solution constructed in Theorem \ref{lm3.1}.  If $\tau \ne \tau_{\max}$, then we extend the result in the usual manner to a larger interval $(0,\tau_1),\, \tau_1 > \tau$, by taking $f(t_0)$ as a new initial value, for some $0<  t_0  < \tau$. Provided the theorem holds on $(0,\tau)$, we know that $f(t_0)\in D(G_{0,m}^{(\beta)}) \subset D(G_{0,p}^{(\beta)})$  and we repeat the proof on $(t_0, \tau_1)$ which overlaps with $(0,\tau)$ on an open interval and thus the theorem is valid on $(0,\tau_1)$. Continuing in this manner, we eventually reach $\tau_{\max}$.

As in the proof of Theorem \ref{lm3.1}, we choose $n$ so that $\kappa :=  \frac{m-n}{\gamma_0} \in (0,1).$
Since $\mr f \in D(G_{0,p}^{(\beta)})$, the mild solution $f$ satisfies the integral equation
\begin{align}
f(t) &=   S_{G_{0,m}^{(\beta)}}(t)\mr f + \cl{0}{t} S_{G_{0,p}^{(\beta)}}(t-s)K_{0,m}^{(\beta)} f(s)\md s\nn\\
& = S_{G_{0,p}^{(\beta)}}(t)\mr f + \cl{0}{t} S_{G_{0,p}^{(\beta)}}(s)K_{0,m}^{(\beta)} f(t-s)\md s.
\label{mild}
\end{align}
We first consider the Lipschitz continuity of $f$. Let $t>0$ and $h>0$. We have
\begin{eqnarray*}
&& \frac{f(t+h) - f(t)}{h}  \  =  \  \frac{1}{h}\left(S_{G_{0,m}^{(\beta)}}(h) - I\right) S_{G_{0,m}^{(\beta)}} (t)\mr f \\
&+& \frac{1}{h} \int_{0}^{h} S_{G_{0,p}^{(\beta)}} (t + h - s) K_{0,m}^{(\beta)} f(s)\md s\\&
	+& \frac{1}{h} \int_{0}^{t} S_{G_{0,p}^{(\beta)}} (t - s)(K_{0,m}^{(\beta)} f(s+h)-K_{0,m}^{(\beta)} f(s)) \md s  =:\  I_1(h) + I_2(h) + I_3(h).
	\end{eqnarray*}
Arguing as in Corollary \ref{correg}, we have
\begin{eqnarray*}
&& \left\|\frac{1}{h}\left(S_{G_{0,m}^{(\beta)}} (h) - I\right)\, S_{G_{0,m}^{(\beta)}} (t)\mr f\right\|_{[0,m]} \ = \ \ \left\|\frac{1}{h} S_{G_{0,p}^{(\beta)}}(t)\left(S_{G_{0,p}^{(\beta)}} (h) - I\right) \mr f\right\|_{[0,m]}\\
&& \ \leq   C(\tau)t^{-\kappa} \left\|\frac{1}{h}\left(S_{G_{0,p}^{(\beta)}} (h) - I\right) \mr f\right\|_{[0,p]}
\leq  C_1(\tau)t^{-\kappa} \|G_{0,p}^{(\beta)}\mr f\|_{[0,p]},
\end{eqnarray*}
where $C(\tau) = Ce^{\theta \tau}$, see \eqref{regest}, and $C_1(\tau) = C(\tau)\max\limits_{0\leq t\leq \tau}\|S_{G_{0,p}^{(\beta)}}(t)\|_{[0,p]}$.

 Next, using \eqref{cest2a},
$$
\| S_{G_{0,p}^{(\beta)}} (t + h - s) K_{0,m}^{(\beta)} f(s)\|_{[0,m]}\leq C(\tau) K(\mc U) (t+h-s)^{-\kappa}
$$
and
\begin{align*}
\frac{1}{h} \cl{0}{h} \| S_{G_{0,p}^{(\beta)}} (t + h -s) K_{0,m}^{(\beta)} f(s)\|_{[0,m]} \md s &\leq C(\tau) K(\mc U)\frac{1}{h} \cl{0}{h}(t+h-s)^{-\kappa}\md s \\
\leq C(\tau) K(\mc U)t^{-\kappa} \frac{1}{h} \cl{0}{h}\md s &=  C(\tau) K(\mc U)t^{-\kappa}.
\end{align*}
Finally,  as in \eqref{K1B},
\begin{align*}
&\frac{1}{h} \int_{0}^{t}\| S_{G_{0,p}^{(\beta)}} (t - s)(K_{0,m}^{(\beta)} f(s+h)-K_{0,m}^{(\beta)} f(s))\|_{0,m} \md s\\& \leq M(\tau)L(\mc U) \cl{0}{t}(t-s)^{-\kappa}\frac{\|f(s +h) - f(s)\|_{[0,m]}}{h} \md s. \end{align*}
Thus, for some constants $C_1, C_2$
$$
\frac{\|f(t +h) - f(t)\|_{[0,m]}}{h} \leq \frac{C_1}{t^\kappa} + C_2\cl{0}{t}(t-s)^{-\kappa}\frac{\|f(s +h) - f(s)\|_{[0,m]}}{h} \md s
$$
and, by the Gronwall--Henry inequality \cite[Lemma 7.1]{BLL}, for some constant $C_3$,
\begin{equation}
\frac{\|f(t +h) - f(t)\|_{[0,m]}}{h} \leq C_3t^{-\kappa}.
\label{LC1}
\end{equation}
We note that in the estimates above, we can use the same bounds for $f(t)$ and $f(t+h)$ as the function $t\mapsto f(t+h)$ can be treated as the solution for the initial value $f(h)$ which is in $\mc U$ for $h$ small enough.

To prove the differentiability of $f$,  first we observe that formally differentiating \eqref{mild} gives, for $t \in (0,\tau),$
\begin{equation}
\p_tf(t) = S_{G_{0,p}^{(\beta)}}(t)G_{0,p}^{(\beta)}\mr f  + S_{G_{0,p}^{(\beta)}}(t)K_{0,m}^{(\beta)} \mr f + \cl{0}{t} S_{G_{0,p}^{(\beta)}}(t-s)\partial K_{0,m}^{(\beta)}f(s)\p_sf(s)\md s.
\label{mild1}
\end{equation}
On defining $g(t) :=G_{0,p}^{(\beta)} S_{G_{0,p}^{(\beta)}}(t)\mr f  + S_{G_{0,p}^{(\beta)}}(t)K_{0,m}^{(\beta)} \mr f$ and $P(s)=\partial K_{0,m}^{(\beta)} f(s),$ we see that the derivative of $f$, if it exists, satisfies the linear integral equation
\begin{equation}
w(t) = g(t) + \cl{0}{t} S_{G_{0,m}^{(\beta)}}(t-s)P(s)w(s)\md s.
\label{mild2}
\end{equation}
Now, for $t>0, h>0,$
\begin{align*}
&\|S_{G_{0,p}^{(\beta)}}(t+h)(G_{0,p}^{(\beta)}\mr f +K_{0,m}^{(\beta)} \mr f )- S_{G_{0,p}^{(\beta)}}(t)(G_{0,p}^{(\beta)}\mr f+K_{0,m}^{(\beta)} \mr f )\|_{[0,m]}\\
& = \|S_{G_{0,p}^{(\beta)}}(t)(S_{G_{0,p}^{(\beta)}}(h)-I)(G_{0,p}^{(\beta)}\mr f+K_{0,m}^{(\beta)} \mr f )\|_{[0,m]}\\& \leq C(\tau)t^{-\kappa}\|(S_{G_{0,p}^{(\beta)}}(h)-I)(G_{0,p}^{(\beta)}\mr f+K_{0,m}^{(\beta)} \mr f )\|_{[0,p]}
\end{align*}
and, analogously, for left-hand  limits. Hence $t\mapsto g(t)$ is in  $C((0,\tau), X_{0,m})$ and is $O(t^{-\kappa})$ close to $t=0$. Next, by Lemma \ref{lemdif}, $s\mapsto P(s)$ is a continuous function that takes values in $\mc L(X_{0,m}, X_{0,p})$. Hence, Lemma \ref{leminteq} yields the existence of a solution $w \in C((0,T], X_{0,m})$ to \eqref{mild2} for any $0<T<\tau$, with $\|w(t)\|_{[0,m]} = O(t^{-\kappa})$ as $t\to 0^+$.

Next,  we prove that $f$ is differentiable in $X_{0,m}$ for $0<t<\tau$. From  \eqref{eq3.2}, we obtain
\[
\frac{f(t+h) - f(t)}{h}  -w(t) = J_1(h) + J_2(h) + J_3(h),
\]
where
\begin{eqnarray*}
J_1(h) &:=&  \frac{1}{h}\left(S_{G_{0,p}^{(\beta)}} - I\right) S_{G_{0,p}^{(\beta)}} (t)\mr f - S_{G_{0,p}^{(\beta)}}(t)G_{0,p}^{(\beta)}\mr f, \\
J_2(h) &:=& \frac{1}{h} \int_{0}^{h}\left( S_{G_{0,p}^{(\beta)}} (t + h - s) K_{0,m}^{(\beta)} f(s)  - S_{G_{0,p}^{(\beta)}}(t)K_{0,m}^{(\beta)} \mr f\right)\md s,\\
J_3(h) &:=& \frac{1}{h} \int_{0}^{t} S_{G_{0,p}^{(\beta)}} (t - s)(K_{0,m}^{(\beta)} f(s+h)-K_{0,m}^{(\beta)} f(s)) \md s\\
&& -\cl{0}{t} S_{G_{0,p}^{(\beta)}}(t-s)P(s)w(s)\md s.
\end{eqnarray*}
Clearly $\lim_{h\to 0^+} J_1(h) =0$ by \eqref{247}.  For $J_2(h)$, we take $t > 0$ and $0 \leq s \leq h \leq t/2$. Then
\begin{align*}
&\|S_{G_{0,p}^{(\beta)}} (t + h - s) K_{0,m}^{(\beta)} f(s) - S_{G_{0,p}^{(\beta)}} (t ) K_{0,m}^{(\beta)} \mr f\|_{[0,m]}\\
&\leq \|S_{G_{0,p}^{(\beta)}} (t - s)(S_{G_{0,p}^{(\beta)}} ( h) -I)K_{0,m}^{(\beta)} f(s) \|_{[0,m]}\\&\phantom{xx}+\|S_{G_{0,p}^{(\beta)}} (t - s) (K_{0,m}^{(\beta)} f(s) - K_{0,m}^{(\beta)} \mr f)\|_{[0,m]}\\&\phantom{xx}+\|(S_{G_{0,p}^{(\beta)}} (t - s) -  S_{G_{0,p}^{(\beta)}} (t )) K_{0,m}^{(\beta)}\mr f\|_{[0,m]}
=:\mathcal{J}_1(s,h)+\mathcal{J}_2(s)+\mathcal{J}_3(s).
\end{align*}
Now
\begin{align*}
\mathcal{J}_1(s,h) &\leq Ce^{\theta (t-s)} (t-s)^{-\kappa} \|(S_{G_{0,p}^{(\beta)}} ( h) -I)K_{0,m}^{(\beta)} f(s)\|_{[0,p]}.
\end{align*}
Since $t\mapsto S_{G_{0,p}^{(\beta)}}(t)$ is strongly continuous in $X_{0,p}$, it is uniformly continuous on compact sets of $X_{0,p}$; that is, for any compact set  $\Omega\subset X_{0,p}$ and each $\e>0$, there exists $h_0 > 0$ such that  $\sup_{u\in \Omega} \|S_{G_{0,p}^{(\beta)}}(h)u-u\|_{[0,p]}\leq \e$ for all $0<h<h_0$. Moreover, as  the function $s \mapsto K_{0,m}^{(\beta)} f(s)$ is $X_{0,p}$-continuous for any $X_{0,m}$-continuous function $f$, and the continuous  image of the compact interval $\left[0,\frac{t}{2}\right]$ is compact, we see that for any $\e>0$ there is $h_0<\frac{t}{2}$ such that for all $0<h\leq h_0$
\begin{equation}
\mathcal{J}_1(s,h)\leq \e
\label{J1}
\end{equation}
uniformly in $s \in [0,h_0]$. Similarly, by \eqref{K1B},
\begin{align*}
\mathcal{J}_2(s)&\leq \|S_{G_{0,p}^{(\beta)}} (t - s) (K_{0,m}^{(\beta)} f(s) - K_{0,m}^{(\beta)}) \mr f\|_{[0,m]}\\
& \leq Ce^{\theta (t-s)} (t-s)^{-\kappa}\|K_{0,m}^{(\beta)} f(s) - K_{0,m}^{(\beta)} \mr f\|_{[0,p]}\\&\leq  Ce^{\theta (t-s)} (t-s)^{-\kappa}L(\mc U)\|f(s) - \mr f\|_{[0,m]}
\end{align*}
and for any $\e$ there is $0<h_0<\frac{t}{2}$ such that for any $0\leq s \leq h\leq h_0$ we have
\begin{equation}
\mathcal{J}_2(s)\leq \e.
\label{J2}
\end{equation}
Finally, as with $\mathcal{J}_1$,
\begin{align*}
\mathcal{J}_3(s)&\leq \|S_{G_{0,p}^{(\beta)}} (t - s)(S_{G_{0,p}^{(\beta)}} (s )-I) K_{0,m}^{(\beta)} \mr f\|_{[0,m]} \\
&= Ce^{\theta (t-s)} (t-s)^{-\kappa} \|(S_{G_{p}} ( s) -I)K_{0,m}^{(\beta)} \mr f\|_{[0,p]},
\end{align*}
	 hence $\mathcal{J}_3$ is a continuous function at $0$ and therefore,
\begin{equation}
\lim\limits_{h\to 0^+} \frac{1}{h}\cl{0}{h}\mathcal{J}_3(s)\md s = 0.
\label{J3}
\end{equation}
Summarizing,
$$
\lim\limits_{h\to 0^+} J_2(h) =\lim\limits_{h\to 0^+} \frac{1}{h} \int_{0}^{h} S_{G_{0,p}^{(\beta)}} (t + h - s) K_{0,m}^{(\beta)} f(s) \md s - S_{G_{0,p}^{(\beta)}} (t ) K_{0,m}^{(\beta)} \mr f =0.
$$
Finally, by Lemma \ref{lemdif}, and with $\omega$ defined as in Lemma \ref{rem},
$$
K_{0,m}^{(\beta)} f(s+h)-K_{0,m}^{(\beta)} f(s) - P(s) (f(s+h) - f(s)) =  \omega(f(s+h) - f(s),f(s)).
$$
Now
$$
\frac{\|\omega(f(s+h)\! -\! f(s))\|_{[0,p]}}{h} = \frac{\|\omega(f(s+h)\! -\! f(s),f(s))\|_{[0,p]}}{\|f(s+h)\! - \! f(s)\|_{[0,m]}}\frac{\|f(s+h) \!- \!f(s)\|_{[0,m]}}{h}.
$$
By Lemma \ref{rem}, the function $$(h,s)\mapsto \frac{\|\omega(f(s+h) - f(s), f(s))\|_{[0,p]}}{\|f(s+h) - f(s)\|_{[0,m]}}$$ is continuous on $[0,h_0]\times [0,s']$ for any $s'<s$ and hence it is uniformly continuous. Thus,  for any $\e>0$ there is $h_0$ such that for any $0<h<h_1\leq h_0, s \in [0,s']$
$$
\frac{\|\omega(f(s+h) - f(s), f(s))\|_{[0,p]}}{\|f(s+h) - f(s)\|_{[0,m]}}\leq \e.
$$
Hence, by \eqref{LC1} and \eqref{WLconv}
\begin{align*}
&\left\|\frac{1}{h} \int_{0}^{t}\!\! S_{G_{0,p}^{(\beta)}} (t \!-\! s)(K_{0,m}^{(\beta)} f(s+h)\!-\!K_{0,m}^{(\beta)} f(s)) \md s \!-\!\!\cl{0}{t}\!\! S_{G_{0,p}^{(\beta)}}(t-s)P(s)w(s)\md s\right\|_{[0,m]}\\
&= \cl{0}{t} \left\|S_{G_{0,p}^{(\beta)}} (t - s)\frac{\omega(f(s+h) - f(s),f(s))}{h}\right\|_{[0,m]} \md s \\&\phantom{x}+ \cl{0}{t}\left\|S_{G_{0,p}^{(\beta)}} (t - s)P(s)\left( \frac{f(s+h) - f(s)}{h} - w(s)\right )\right\|_{[0,m]}\md s\\
&\leq C_1\cl{0}{t}  (t - s)^{-\kappa}\left\|\frac{\omega(f(s+h) - f(s),f(s))}{h}\right\|_{[0,p]} \md s\\&\phantom{x} + C_2\cl{0}{t} (t - s)^{-\kappa}\left\| \frac{f(s+h) - f(s)}{h} - w(s)\right\|_{[0,m]}\md s \\
&\leq C_1C_3\e\cl{0}{t}  (t - s)^{-\kappa}s^{-\kappa} \md s + C_2\cl{0}{t} (t - s)^{-\kappa}\left\| \frac{f(s+h) - f(s)}{h} - w(s)\right\|_{[0,m]}\!\!\!\md s \\
&= C_1C_3B(1-\kappa,1-\kappa) \e t^{1-2\kappa} + C_2\!\!\cl{0}{t} (t - s)^{-\kappa}\left\| \frac{f(s+h) - f(s)}{h} - w(s)\right\|_{[0,m]}\!\!\!\md s.
\end{align*}
 Since for small $t$ we have $t^{1-2\kappa}\leq t^{-\kappa}$, it follows that, on any time interval $(0,s')$ where  $s' < s$, and for any $\e>0$, there is $h_0$ such that for any $0<h<h_0$
 \begin{align*}
 &\left\| \frac{f(t+h) - f(t)}{h} - w(t)\right\|_{[0,m]}\\
 &\phantom{xx}\leq \e t^{-\kappa}C_5 + C_2 \cl{0}{t} (t - s)^{-\kappa}\left\| \frac{f(s+h) - f(s)}{h} - w(s)\right\|_{[0,m]}\md s
 \end{align*}
 and thus, by \cite[Lemma 3.2]{Banasiak2019},
 $$
 \left\| \frac{f(t+h) - f(t)}{h} - w(t)\right\|_{[0,m]}\leq \e t^{-\kappa} C_6.
 $$
 Hence the right-hand derivative of $f$ exists on $(0,\tau)$, and  satisfies \eqref{mild1}. As in the proof of Theorem \ref{lm3.1}, the right-hand side of \eqref{mild1} is continuous on $(0,s)$ and thus the left-hand  derivative also exists.
 Hence $f \in C^1((0,\tau), X_{0,m})$.

 To show that $f(t) \in D(G_{0,p}^{(\beta)})$ for $t>0$, we evaluate
 \begin{align*}
&	   \frac{1}{h}(S_{G_{0,p}^{(\beta)}} (h) - I)f(t) = \frac{1}{h}(S_{G_{0,p}^{(\beta)}} (h) - I) S_{G_{0,p}^{(\beta)}} (t)\mr f \\& + \frac{1}{h} \cl{0}{t} S_{G_{0,m}^{(\beta)}} (t  - s) K_{0,m}^{(\beta)} f(s+ h) \md s
	- \frac{1}{h} \cl{0}{t} S_{G_{0,p}^{(\beta)}} (t - s)K_{0,m}^{(\beta)} f(s) \md s\\
&=   \frac{1}{h}S_{G_{0,p}^{(\beta)}} (t)(S_{G_{p}} (h) - I) \mr f + \frac{1}{h} \int_{0}^{h} S_{G_{0,p}^{(\beta)}} (t + h - s) K_{0,m}^{(\beta)} f(s)\md s \\
&\phantom{x}- \frac{1}{h} \int_{t-h}^{t} S_{G_{0,p}^{(\beta)}} (t + h - s) K_{0,m}^{(\beta)} f(s)\md s\\&\phantom{x}
	+ \frac{1}{h} \int_{0}^{t} S_{G_{0,p}^{(\beta)}} (t - s)(K_{0,m}^{(\beta)} f(s+h)-K_{0,m}^{(\beta)} f(s)) \md s \\&
 =: L_1(h) + L_2(h) + L_3(h) +L_4(h).
	\end{align*}
 Using again \eqref{247}, $L_1(h) \to  S_{G_{0,m}^{(\beta)}}(t) G_{0,m}^{(\beta)}\mr f$ in $X_{0,m}$ for $t>0.$ Also, as above,
 $$
 \lim\limits_{h\to 0^+} L_2(h) = S_{G_{0,m}^{(\beta)}} (t ) K_{0,m}^{(\beta)} \mr f
 $$
 and
 $$
 \lim\limits_{h\to 0^+} L_4(h) = \cl{0}{t} S_{G_{0,p}^{(\beta)}}(t-s)\partial K_{0,m}^{(\beta)}f(s)\p_sf(s)\md s.
 $$
 Then, in the same way as for $L_2$, we have
  $$
 \lim\limits_{h\to 0^+} L_3(h) =  -K_{0,m}^{(\beta)} f(t)
 $$
 in $X_{0,p}$. Hence $f(t) \in D(G_{0,p}^{(\beta)})$ for $t>0$ and
 \begin{align}
 G_{0,p}^{(\beta)} f(t) &=  - K_{0,m}^{(\beta)} f(t) \nn\\
 &\phantom{x}+S_{G_{0,p}^{(\beta)}}(t) G_{0,p}^{(\beta)}\mr f + S_{G_{0,p}^{(\beta)}} (t ) K_{0,m}^{(\beta)} \mr f +\!\cl{0}{t}\! S_{G_{0,p}^{(\beta)}}(t\!-\!s)\partial K_{0,m}^{(\beta)}f(s)\p_sf(s)\md s\nn\\
 &= - K_{0,m}^{(\beta)} f(t) + \p_t f(t).\label{classsol}
 \end{align}
\end{proof}

\subsection{Global solvability}

To establish the existence of global (in time) solutions to the growth C-F equation we must impose the more restrictive condition
\begin{equation}
k(x,y) \leq k_0 (1+ x^\alpha+y^\alpha)
\label{kass1}
\end{equation}
on the coagulation kernel. As in \eqref{kass}, $k_0$ is a positive constant and $0 < \alpha<\gamma_0$, where $\gamma_0$ is given in \eqref{assa1}.  Also, the inclusion of the term $a_1(x) = \beta(1+ x^\alpha)$ being required only to prove the nonnegativity of mild solutions in Theorem \ref{lm3.1}, we now set $\beta = 0$, in which case, from \eqref{genrep} and Theorem \ref{th3.4}, there exists a unique solution $f$ to
\begin{equation}\label{feq0}
\frac{d}{dt}f(t) = T_{0,p}^0f(t) + A_{0,p}f(t) + B_{0,p}f(t) + K_{0,m}f(t),\ \quad  f(0) = \mr f \in X_{0,m}\cap D(G_{0,p}),
\end{equation}
in $C([0, \tau_{\max}), X_{0,m}) \cap C^{1}((0, \tau_{\max}), X_{0,m})\cap C((0, \tau_{\max}), D(G_{0,p}))$,
where $K_{0,m} := K_{0,m}^{(0)}$ and $G_{0,p} := G_{0,p}^{(0)}$. We emphasize that, once $\alpha$ is given, we can use an arbitrary $p>\max\{1,l\}$ and then  take $m = p+\alpha$.
\begin{theorem}\label{th3.3}
Let all the assumptions of Theorem \ref{lm3.1} hold, but with \eqref{kass} replaced by \eqref{kass1}. If either \begin{description}
\item {(i)} there are constants $m_0$ and $m_1$ such that $(n_0(x)-1)a(x) \leq m_0+m_1x$, for all $x \geq 0$,
where $n_0$ is defined by \eqref{nmy}, or
\item {(ii)} $r(x)\leq \ti r x$, for all $x > 0$ (i.e. $r_0 = 0$ in  \eqref{fmlras}),
\end{description}
then the solutions of Theorem \ref{lm3.1} are global in time.
\end{theorem}
\begin{proof} The proof follows similar lines to that of \cite[Theorem 5.1]{Ban2020},  but some of the technicalities are slightly different. The hypothesis of the theorem guarantee the existence of a mild solution to
\begin{equation}
\frac{d}{dt}f(t) = T^0_{0,m} f(t)+ A_{0,m} f(t) +B_{0,m} f(t) + K_{0,m} f(t),\ \quad t\in (0,\tau_{\max}).
\label{feq}
\end{equation}
Using the classical identities and estimates, \cite[Eqn. (8.1.22) \& Lemma 7.4.2]{BLL}, and \eqref{kass1},
\begin{align}
\cl{0}{\infty}x^i\mc Kf(x)\md x &=
\frac{1}{2}\cl{0}{\infty}\cl{0}{\infty}((x+y)^i-x^i-y^i)k(x,y)f(x)f(y)\md x\md y, \nn\\
&\leq \frac{C_ik_0}{2}\cl{0}{\infty}\cl{0}{\infty}(yx^{i-1}+xy^{i-1})(1+ x^\alpha+y^\alpha)f(x)f(y)\md x\md y,\nn\\
&\leq K_i (\|f\|_{[1]}\|f\|_{[i-1]} + \|f\|_{[1]} \|f\|_{[\alpha+i-1]} + \|f\|_{[\alpha+1]}\|f\|_{[i-1]}),
\label{coag1}
\end{align}
for $i\geq 1$, where $K_i=C_ik_0$ and the norms are defined by \eqref{norms}.   First we consider  $\mr f$ to be a $C^\infty(\mbb R_+)$ function with bounded support. Then $\mr f \in D(G_{0,i})$ for any $i$  and, if additionally $i>\max\{1,l\}$, then, by Theorem \ref{th3.4}, the corresponding solution $(0,\tau_{\max})\ni t\mapsto f(t)= f(t,\mr f)$ is differentiable in any space $X_{i}$. Hence, using \eqref{subfuncta'} (with $a_1(x) \equiv 0$), and recalling that $M_m(t)$ is given by \eqref{Moments},  we obtain
\begin{align}
\frac{d}{dt} M_0(t) &=    -\cl{0}{\infty}N_0(x)a(x)f(x,t)\md x-\frac{1}{2}\cl{0}{\infty}\cl{0}{\infty}k(x,y)f(x,t)f(y,t)\md x\md y   \label{M0}\\
\frac{d}{dt} M_1(t) &=  \cl{0}{\infty}r(x) f(x,t) \md x\label{M1}\\
\frac{d}{dt} M_i(t) &=  \cl{0}{\infty}\left(ir(x)x^{i-1}  -N_i(x)a(x)\right)f(x,t)  \mdm{d}x \nn\\&\phantom{xx}+ \frac{1}{2}\cl{0}{\infty}\int_{0}^{\infty}((x+y)^i-x^i-y^i)k(x,y)f(x,t)f(y,t)\md x\md y, \quad i>1.\label{feco1}
\end{align}
As pointed out earlier, $N_0(y) = 1-n_0(y) <0$ due to \eqref{Nm}.

Let us consider first the term in \eqref{feco1} containing $N_i$ and recall that $a_0,\gamma_0$ and $x_0$ are the constants given in \eqref{assa1}. Similarly to \eqref{bmom1} (see also  \cite[Theorem 2.2]{Ban2020}),  we have that if $N_{m_0}(x)/x^{m_0} \geq \delta'_{m_0}$ holds for some $m_0 > 1,$ $\delta'_{m_0}$ and $x\geq x_0$, then  for any $i >1$ there is $\delta'_i>0$ such that $N_i(x)/x^i \geq \delta'_i>0$ for any $x\geq x_0$. Hence,
  \begin{align}
  -\cl{0}{\infty}N_i(x)a(x)f(x)\md x &=  -\cl{0}{x_0}a(x) N_i(x)f(x)\md x - \cl{x_0}{\infty} a(x)f(x) x^i \frac{N_i(x)}{x^i}\md x\nn\\
&\leq   - \delta'_i\cl{0}{\infty} a(x)f(x) x^i \md x  + \delta'_i\cl{0}{x_0}a(x) x^if(x)\md x\nn\\
&\leq - \delta_i \|f\|_{[i+\gamma_0]}  + \nu_i \| f\|_{[i]},\label{fragest}
\end{align}
 where $\delta_i= \delta'_i a_0$ and  $\nu_i = \delta_i\text{ess}\sup_{0\leq x\leq x_0} a(x)$.
   First, let us consider an integer $i \geq 2$. Then, from \eqref{fragest}, together with \eqref{M0} and \eqref{M1}, \begin{align}
 \frac{d}{dt} M_0(t) &\leq    \cl{0}{\infty}(n_0(x)-1)a(x)f(x,t)\md x   \nn\\
\frac{d}{dt} M_1(t) &\leq   \ti r M_0(t) + \ti r M_1(t)\nn\\
 \frac{d}{dt} M_{i}(t) &\leq \ti rM_{i-1}(t) + (\nu_i+\ti r) M_{i}(t) - \delta_i M_{i+\gamma_0}(t) \nn \\
 & \quad + K_i(M_{1}(t)M_{i-1}(t) + M_{1}(t) M_{\alpha+i-1}(t) + M_{\alpha+1}(t)M_{i-1}(t)).\label{firstmom}
\end{align}
To simplify \eqref{firstmom}, we use  the following auxiliary inequalities.  For  $i \geq 2$ and  $1\leq r \leq i-1,$ we apply H\"{o}lder's inequality with $p=\gamma_0/\alpha$ and $q =\gamma_0/(\gamma_0-\alpha)$ to  obtain
\begin{align}
\|f\|_{[r+\alpha]} &= \int_0^\infty x^r x^\alpha f(x) \md x = \int_0^1 x^r x^\alpha f(x) \md x + \int_1^\infty x^r x^\alpha f(x) \md x\nn\\
&\leq c_\alpha \int_0^1 x f(x) \md x  + \int_1^\infty x^{(i-1)/q}f^{1/q}(x)x^{r-\frac{i-1}{q}} x^{\frac{\gamma_0}{p}} f^{1/p}(x) \md x\nn\\
&\leq c_\alpha\| f\|_{[1]}  + \left(\int_0^\infty x^{i-1}f(x)\md x\right)^{\frac{1}{q}}\left(\int_1^\infty x^{pr-\frac{p(i-1)}{q}} x^{\gamma_0} f(x) \md x\right)^{\frac{1}{p}}\nn\\
&\leq c_\alpha\| f\|_{[1]}  + \|f\|_{[i-1]}^{\frac{\gamma_0-\alpha}{\gamma_0}}\|f\|_{[i+\gamma_0]}^{\frac{\gamma_0}{\alpha}}, \label{wl866}
\end{align}
where we used the fact that for $1\leq r\leq i-1$
$$
pr-\frac{p(i-1)}{q} = \frac{\gamma_0}{\alpha}r - \left(\frac{\gamma_0}{\alpha}-1\right)(i-1) \leq i-1 < i$$  and hence
$$
x^{pr-\frac{p(i-1)}{q}} \leq x^i, \qquad x \in [1,\infty).
$$
Then  Young's inequality gives, for any $\e > 0$,
\begin{align}
\|f\|_{[i+\alpha -1]}\|f\|_{[1]} &\leq c_\alpha\|f\|^2_{[1]} + \|f\|_{[1]}\|f\|_{[i-1]}^{\frac{\gamma_0-\alpha}{\gamma_0}}\|f\|_{[i+\gamma_0]}^{\frac{\gamma_0}{\alpha}}\nn\\
&\leq c_\alpha\|f\|^2_{[1]} + \|f\|_{[1]}\left(\frac{\gamma_0-\alpha}{\gamma_0}\e^{\frac{\gamma_0}{\alpha-\gamma_0}}\|f\|_{[i-1]} + \frac{\alpha}{\gamma_0}\e^{\frac{\gamma_0}{\alpha}}\|f\|_{[i+\gamma_0]}\right)
\end{align}
and
\begin{align}
\|f\|_{[i-1]}\|f\|_{[1+\alpha]} &\leq c_\alpha\|f\|_{[1]} \|f\|_{[i-1]}+ \|f\|_{[i-1]}^{\frac{2\gamma_0-\alpha}{\gamma_0}}\|f\|_{[i+\gamma_0]}^{\frac{\gamma_0}{\alpha}}\nn\\
&\leq c_\alpha\|f\|_{[1]} \|f\|_{[i-1]}\! +\! \left(\!\frac{\gamma_0-\alpha}{\gamma_0}\e^{\frac{\gamma_0}{\alpha-\gamma_0}}\|f\|_{[i-1]}^{\frac{2\gamma_0-\alpha}{\gamma_0-\alpha}} + \frac{\alpha}{\gamma_0}\e^{\frac{\gamma_0}{\alpha}}\|f\|_{[i+\gamma_0]}\!\!\right).
\end{align}
We now apply these inequalities to the solution $t\mapsto f(t)$, transforming the last inequality in \eqref{firstmom} into
\begin{align}
&\frac{d}{dt} M_{i}(t) \leq \ti rM_{i-1}(t) + (\nu_i+\ti r) M_{i}(t) - \delta_i M_{i+\gamma_0}(t) \nn \\
 & \quad + K_i\left(\phantom{\frac{a}{b}}\!\!\!\!\!M_{1}(t)M_{i-1}(t) + c_\alpha M^2_{1}(t) \right.\nn\\
  &\quad \left.+ M_{1}(t)\left(\frac{\gamma_0-\alpha}{\gamma_0}\e^{\frac{\gamma_0}{\alpha-\gamma_0}}M_{i-1}(t) + \frac{\alpha}{\gamma_0}\e^{\frac{\gamma_0}{\alpha}}M_{i+\gamma_0}(t)\right)\right. \nn\\
  &\quad + \left. c_\alpha M_1(t)M_{i-1}(t) + \left(\frac{\gamma_0-\alpha}{\gamma_0}\e^{\frac{\gamma_0}{\alpha-\gamma_0}}M_{i-1}^{\frac{2\gamma_0-\alpha}{\gamma_0-\alpha}}(t) + \frac{\alpha}{\gamma_0}\e^{\frac{\gamma_0}{\alpha}}M_{i+\gamma_0}(t)\right)\right).\label{mom2}
\end{align}
There remains the problem that the estimates derived above require some control of $M_1(t)$. This presents no difficulties for the standard, mass-conserving C-F models, as then $M_{1}(t) = \|\mr f\|_{[1]}$  for all $t \in [0, \tau_{\max})$. Here, however, the second inequality of \eqref{firstmom} shows that $M_1(t)$ is coupled with $M_0(t)$, and the latter in general depends on higher order moments. There are two easy ways to remedy this situation, related to assumptions (i) and (ii), respectively. If (i) is satisfied,  then
\begin{align*}
\frac{d}{dt} M_0(t) &\leq    \cl{0}{\infty}(n_0(x)-1)a(x)f(x,t)\md x \leq m_0 M_0(t) + m_1 M_1(t),\nn\\
\frac{d}{dt} M_1(t) &\leq    r_0 M_0(t) +  r_1M_1(t),
\end{align*}
which yields $M_0(t) \leq \mr M_0 e^{\mu t}$ and $M_1(t) \leq \mr M_1 e^{\mu t}$ for some constant $\mu$ and thus neither moment blows up in finite time. If (ii) is satisfied, then obviously $M_1(t) \leq \mr M_1 e^{\ti r t}$ and the inequalities for the moments of order greater than one become decoupled from the zeroth order moment.   In both cases $M_1(t)\leq M_{1,\tau_{\max}}$ on $[0,\tau_{max})$
 and, by choosing $\e$ so that $\frac{\alpha}{\gamma_0}\e^{\frac{\gamma_0}{\alpha}} K_i(M_{1,\tau_{\max}} +1)\leq \delta_i$, we see that there are positive constants $D_{0,i}, D_{1,i}, D_{2,i}, D_{3,i}$ such that (\ref{mom2}) can be written as
\begin{equation}
 \frac{d}{dt} M_{i}(t) \leq D_{0,i}+ D_{1,i} M_{i}(t) + D_{2,i} M_{i-1}(t) + D_{3,i} M_{i-1}^{\frac{2\gamma_0-\alpha}{\gamma_0-\alpha}}(t),   \label{firstmom1}
\end{equation}
for $t\in [0,\tau_{\max})$. In particular, for $i = 2$ we obtain
\begin{equation}
 \frac{d}{dt} M_{2}(t) \leq D_{0,2}+ D_{1,2} M_{2}(t) + D_{2,2} M_{1}(t) + D_{3,2} M_{1}^{\frac{2\gamma_0-\alpha}{\gamma_0-\alpha}}(t), \label{firstmom2}
\end{equation}
for $ t\in [0,\tau_{\max})$, and thus $t \mapsto M_{2}(t)$ is bounded on $[0, \tau_{\max})$. Then we can use \eqref{firstmom1} to proceed inductively to establish the  boundedness of $t \mapsto M_{i}(t)$ for any integer $i$ (for the chosen initial condition). Further, since for any $i>1$ we have $x^i \leq x$ for $x \in [0,1]$ and $x^i \leq x^{\lfloor i\rfloor +1}$ for $x \geq 1$, then
 $$
 \|f\|_{[i]} \leq \|f\|_{[1]} + \|f\|_{[\lfloor i\rfloor +1]},
 $$
 and we find that all moments of the solution of order $i\geq 1$ are bounded on the maximal interval of its existence.

 It remains to prove that $t\mapsto M_0(t)$ is bounded on $[0,\tau_{\max})$ (in case (ii)). Let us fix an integer $i>\max\{1,l\}.$ Using the fact that
$$
\cl{0}{\infty} \mc K f(x,t)\md x \leq 0$$
and, from \eqref{PhPr005},
\begin{align}
\cl{0}{\infty} \mc F f(x,t)\md x &\leq \cl{0}{\infty} (n_0(y)-1)a(y) f(y,t)\md y \leq 2b_0\cl{0}{\infty} a(y) f(y,t) w_i(y)\md y \nn\\
&\leq \ti a\cl{0}{x_0}  f(y,t) \md y + 2b_0 R(t),
\label{eqP1}
\end{align}
on $[0,\tau_{\max}),$ where $\ti a = 2b_0 \text{ess}\sup_{y\in [0,x_0]}a(y)w_i(y) $, for the zeroth moment we have
$$
\frac{d}{dt}M_0(t) \leq \ti a M_0(t) + 2b_0 R(t),
$$
where we denoted
$$
R(t)= \cl{x_0}{\infty} a(x) f(x,t)w_i(x)\md x.
$$
Hence
\begin{equation}
M_{0}(t) \leq e^{\ti a t}\left(\|\mr f\|_{[0]} + 2b_0\cl{0}{t} R(s)\md s\right).
\label{Pt0}
\end{equation}
We have the estimate
\begin{align}
 \cl{0}{t}R(s) \md s&=\cl{0}{t}\cl{x_0}{\infty} a(x) f(x,s)w_i(x)\md x\md s \leq (1+x_0^{-i})\cl{0}{t}\cl{x_0}{\infty} a(x) f(x,s)x^i\md x\md s.
\label{Pt1}
\end{align}

Now, as in  \eqref{fragest},
 \begin{align}
 \cl{0}{\infty} \mc F f(x)x^i\md x & = -\cl{0}{\infty}N_i(x)a(x)f(x)\md x \leq  - \cl{x_0}{\infty} a(x)f(x) x^i \frac{N_i(x)}{x^i}\md x\nn\\
&\leq   - \frac{\delta'_i}{2}\cl{x_0}{\infty} a(x)f(x) x^i \md x  - \frac{\delta'_i}{2}\cl{0}{\infty} a(x)f(x) x^i \md x  + \frac{\delta'_i}{2}\cl{0}{x_0}a(x) x^if(x)\md x\nn\\
&\leq  - \frac{\delta'_i}{2}\cl{x_0}{\infty} a(x)f(x) x^i \md x - \frac{\delta_i}{2} \|f\|_{[i+\gamma_0]}  + \nu_i \| f\|_{[i]},
 \label{Fest1}
 \end{align}
 where $\delta_i$ and  $\nu_i$ were defined previously.
Now, knowing that all lower order moments are finite on $[0,\tau_{\max})$ and selecting  $\e$ so that $\frac{\alpha}{\gamma_0}\e^{\frac{\gamma_0}{\alpha}} K_i(M_{1,\tau_{\max}} +1)\leq \frac{\delta_i}{2},$  we can write \eqref{firstmom1} as
\begin{equation}
 \frac{d}{dt} M_{i}(t) \leq - \frac{\delta'_i}{2}\cl{x_0}{\infty} a(x)f(x,t) x^i \md x  + D_{0,i}+ D_{1,i} M_{i}(t) + \Theta(t),    \label{firstmom1'}
\end{equation}
 where $\Theta$ is bounded on $t\in [0,\tau_{\max})$. This can be re-written as
 \begin{align*}
 &\frac{d}{dt} \left(M_{i}(t)+ \frac{\delta'_i}{2}\cl{0}{t}\cl{x_0}{\infty} a(x)f(x,s) x^i \md x \md s\right) \leq    D_{0,i}+ D_{1,i} M_{i}(t) + \Theta(t)    \\
 &\phantom{xxx}\leq  D_{0,i}+ D_{1,i} \left(M_{i}(t)+ \frac{\delta'_i}{2}\cl{0}{t}\cl{x_0}{\infty} a(x)f(x,s) x^i \md x \md s\right) + \Theta(t).
\end{align*}
 Denoting
 $$
 \Phi(t) = M_{i}(t)+ \frac{\delta'_i}{2}\cl{0}{t}\cl{x_0}{\infty} a(x)f(x,s) x^i \md x \md s
 $$
 and  integrating,
\begin{align*}
\Phi(t) &\leq e^{D_{1,i} t}\left(\Phi(0) + \frac{D_{0,i}}{D_{1,i}}(1-e^{-D_{1,i}t}) + \cl{0}{t}\Theta(s)e^{-D_{1,i} s}\md s\right)
\end{align*}
and we see that neither $\Phi,$ nor
$$
t\mapsto \cl{0}{t}\cl{x_0}{\infty} a(x) f(x,s)x^i\md x\md s
$$
can blow up at $t=\tau_{\max}$. Hence, by \eqref{Pt1} and \eqref{Pt0}, neither can $t\mapsto M_0(t)$.

This shows that solutions emanating from compactly supported differentiable initial conditions are global in time. Consider now $\mr f\in X_{0,m,+}$  and a sequence of such regular initial conditions $(\mr f_k)_{k\geq 1}$ approximating $\mr f$ and assume that  $t\to f(t, \mr f)$ has a finite time blow up at $\tau_{\max}$. By the moment estimates above, the bounds of $\|f(t, \mr f_k)\|_{[0,m]}$ over any finite time interval depend continuously on $\mr f_k$ and thus are uniform in $k$ on $[0,\tau_{\max}]$. On the other hand,
 there is a sequence $(t_n)_{n\geq 1}$ such that $t_n\to \tau_{\max}, n\to \infty,$ and $\|f(t_n, \mr f)\|_{[0,m]}$ is unbounded; that is, the distance between $f(t_n,\mr f)$ and all $f(t_n,\mr f_k)$ becomes arbitrarily large. This contradicts the continuous dependence of solutions on the initial conditions following, on each $[0,t_n]$, from the Gronwall-- Henry inequality, \cite[Lemma 3.2]{Banasiak2019}),   see also \cite[Theorem 8.1.1]{BLL}.
\end{proof}
\begin{remark} The additional restrictions in Theorem \ref{th3.3} are due to the fact that, in the general case, we cannot control the production of particles; that is, the zeroth moment. In principle, there is a positive feedback loop in which $M_0$ contributes to $M_1$ which, in turn, amplifies, in a nonlinear way,  higher order moments that determine the rate of growth of $M_0$. The adopted assumptions, which postulate that either $M_0$ is controlled by $M_1,$ or that the evolution of mass is not influenced by other mechanisms ($r_0\neq 0$ implies that there is a production of mass independent of the existing one), although technical, seem to be the simplest ones that break this cycle.  We do not claim that these assumptions are optimal but at present we do not have any examples of a finite time blow up of solutions in this setting. It is, however, worthwhile to note that there are known cases of a finite time blow up of solutions to growth--fragmentation--coagulation equations even with bounded coagulation kernels but with  the renewal boundary condition,\cite{Bana12c}.
\end{remark}


\begin{thebibliography}{10}

\bibitem{abramowitz1964handbook}
M.~Abramowitz and I.~A. Stegun.
\newblock {\em Handbook of mathematical functions: with formulas, graphs, and
  mathematical tables}, volume~55.
\newblock Courier Corporation, 1964.

\bibitem{AizBak}
M.~Aizenman and T.~A. Bak.
\newblock Convergence to equilibrium in a system of reacting polymers.
\newblock {\em Comm. Math. Phys.}, 65(3):203--230, 1979.

\bibitem{BaT01}
J.~Banasiak.
\newblock On an extension of the {K}ato-{V}oigt perturbation theorem for
  substochastic semigroups and its application.
\newblock {\em Taiwanese J. Math.}, 5(1):169--191, 2001.

\bibitem{Bana12c}
J.~Banasiak.
\newblock Blow-up of solutions to some coagulation and fragmentation equations
  with growth.
\newblock {\em Discrete Contin. Dyn. Syst.}, (Dynamical systems, differential
  equations and applications. 8th AIMS Conference. Suppl. Vol. I):126--134,
  2011.

\bibitem{Bana12a}
J.~Banasiak.
\newblock Transport processes with coagulation and strong fragmentation.
\newblock {\em Discrete Contin. Dyn. Syst. Ser. B}, 17(2):445--472, 2012.

\bibitem{Ban2020}
J.~Banasiak.
\newblock Global solutions of continuous coagulation-fragmentation equations
  with unbounded coefficients.
\newblock {\em Discrete Contin. Dyn. Syst. Ser. S}, 2020.

\bibitem{BaAr}
J.~Banasiak and L.~Arlotti.
\newblock {\em Perturbations of positive semigroups with applications}.
\newblock Springer Monographs in Mathematics. Springer-Verlag London, Ltd.,
  London, 2006.

\bibitem{Banasiak2019}
J.~{Banasiak}, L.~O. {Joel}, and S.~{Shindin}.
\newblock The discrete unbounded coagulation-fragmentation equation with
  growth, decay and sedimentation.
\newblock {\em Kinet. Relat. Models}, 12(5):1069--1092, 2019.
\newblock doi:10.3934/krm.2019040, arXiv:1809.00046.

\bibitem{Banasiak2018}
J.~{Banasiak}, L.~O. {Joel}, and S.~{Shindin}.
\newblock Long term dynamics of the discrete growth-decay-fragmentation
  equation.
\newblock {\em J. Evol. Equ.}, 19:771--802, 2019.

\bibitem{BaLa09}
J.~Banasiak and W.~Lamb.
\newblock Coagulation, fragmentation and growth processes in a size structured
  population.
\newblock {\em Discrete Contin. Dyn. Syst. Ser. B}, 11(3):563--585, 2009.

\bibitem{BaLa12a}
J.~Banasiak and W.~Lamb.
\newblock Analytic fragmentation semigroups and continuous
  coagulation-fragmentation equations with unbounded rates.
\newblock {\em J. Math. Anal. Appl.}, 391(1):312--322, 2012.

\bibitem{BaLa12b}
J.~Banasiak and W.~Lamb.
\newblock The discrete fragmentation equation: semigroups, compactness and
  asynchronous exponential growth.
\newblock {\em Kinet. Relat. Models}, 5(2):223--236, 2012.

\bibitem{BLL13}
J.~Banasiak, W.~Lamb, and M.~Langer.
\newblock Strong fragmentation and coagulation with power-law rates.
\newblock {\em J. Engrg. Math.}, 82:199--215, 2013.

\bibitem{BLL}
J.~Banasiak, W.~Lamb, and P.~Lauren{\c{c}}ot.
\newblock {\em Analytic methods for coagulation-fragmentation models}, volume
  1\&2.
\newblock CRC Press, 2019.

\bibitem{Ber2019}
E.~Bernard and P.~Gabriel.
\newblock Asynchronous exponential growth of the growth-fragmentation equation
  with unbounded fragmentation rate.
\newblock {\em Journal of Evolution Equations}, pages 1--27, 2019.
\newblock https://doi.org/10.1007/s00028-019-00526-4.

\bibitem{bobrowski2016convergence}
A.~Bobrowski.
\newblock {\em Convergence of one-parameter operator semigroups}, volume~30.
\newblock Cambridge University Press, 2016.

\bibitem{DoGa10}
M.~Doumic~Jauffret and P.~Gabriel.
\newblock Eigenelements of a general aggregation-fragmentation model.
\newblock {\em Math. Models Methods Appl. Sci.}, 20(5):757--783, 2010.
\bibitem{Lun}
A.~Lunardi.
\newblock {\em Analytic semigroups and optimal regularity in parabolic
  problems}, volume~16 of {\em Progress in Nonlinear Differential Equations and
  their Applications}.
\newblock Birkh\"auser Verlag, Basel, 1995.


\bibitem{MLM97a}
D.~J. McLaughlin, W.~Lamb, and A.~C. McBride.
\newblock An existence and uniqueness result for a coagulation and
  multiple-fragmentation equation.
\newblock {\em SIAM J. Math. Anal.}, 28(5):1173--1190, 1997.

\bibitem{LML}
D.~J. McLaughlin, W.~Lamb, and A.~C. McBride.
\newblock A semigroup approach to fragmentation models.
\newblock {\em SIAM J. Math. Anal.}, 28(5):1158--1172, 1997.

\bibitem{Pa}
A.~Pazy.
\newblock {\em Semigroups of linear operators and applications to partial
  differential equations}, volume~44 of {\em Applied Mathematical Sciences}.
\newblock Springer-Verlag, New York, 1983.

\bibitem{PerTr}
B.~Perthame.
\newblock {\em Transport equations in biology}.
\newblock Frontiers in Mathematics. Birkh\"auser Verlag, Basel, 2007.

\bibitem{Schu40}
T.~Schumann.
\newblock Theoretical aspects of the size distribution of fog particles.
\newblock {\em Q. J. Roy. Meteorol. Soc.}, 66:195--207, 1940.

\bibitem{Scot68}
W.~T. Scott.
\newblock Analytic studies of cloud droplet coalescence i.
\newblock {\em J. Atmos. Sci.}, 25:54--65, 1968.
\bibitem{SY}
G.~R. Sell and Y.~You.
\newblock {\em Dynamics of evolutionary equations}, volume 143.
\newblock Springer Science \& Business Media, 2013.


\bibitem{Smoluch}
M.~v. {Smoluchowski}.
\newblock {Drei Vortrage uber Diffusion, Brownsche Bewegung und Koagulation von
  Kolloidteilchen}.
\newblock {\em Zeitschrift fur Physik}, 17:557--585, 1916.

\bibitem{Smoluch17}
M.~v. {Smoluchowski}.
\newblock Versuch einer mathematischen theorie der koagulationskinetik
  kolloider l{\"o}sungen.
\newblock {\em Zeitschrift fuer physikalische Chemie}, 92:129 -- 168, 2010.

\bibitem{vizi89}
R.~D. Vigil and R.~M. Ziff.
\newblock On the stability of coagulation-fragmentation population balances.
\newblock {\em J. Colloid Interface Sci.}, 133(1):257--264, 1989.

\end{thebibliography}
\end{document}